\documentclass[12pt,letterpaper]{article}

\usepackage{amssymb,amsfonts,amscd,amsthm}
\usepackage[all,arc]{xy}
\usepackage{enumerate, bbm}
\usepackage{mathrsfs}
\usepackage{mathtools}
\usepackage{hyperref}
\usepackage[usenames, dvipsnames]{color}
\usepackage{graphicx}
\usepackage{enumitem} 
\usepackage{caption}
\usepackage{subcaption}
\graphicspath{ {TexImages/} }

\newtheorem{thm}{Theorem}[section]
\newtheorem{cor}[thm]{Corollary}
\newtheorem{lem}[thm]{Lemma}
\newtheorem{prop}[thm]{Proposition}
\newtheorem{con}[thm]{Conjecture}
\newtheorem{claim}[thm]{Claim}
\newtheorem{quest}[thm]{Question}
\newtheorem{prob}[thm]{Problem}
\newtheorem{conj}[thm]{Conjecture}

\hypersetup{
colorlinks,
citecolor=blue,
filecolor=blue,
linkcolor=blue,
urlcolor=blue,
linktocpage
}

\setenumerate[1]{label=\thesection.\arabic*.} 
\setenumerate[2]{label*=\arabic*.} 
\setlist[enumerate]{itemsep=2ex, topsep=2ex} 
\setlist[itemize]{itemsep=2ex, topsep=2ex}

\newcommand{\Z}{\mathbb{Z}}
\newcommand{\N}{\mathbb{N}}
\newcommand{\ep}{\epsilon}
\newcommand{\del}{\delta}
\newcommand{\Del}{\Delta}
\renewcommand{\l}{\left}
\renewcommand{\r}{\right}
\newcommand{\half}{\frac{1}{2}}
\newcommand{\sub}{\subseteq}
\newcommand{\ignore}[1]{}

\newcommand{\tr}[1]{\textrm{#1}}
\newcommand{\rec}[1]{\frac{1}{#1}}
\newcommand{\f}[2]{\frac{#1}{#2}}
\newcommand{\floor}[1]{\lfloor #1\rfloor}
\newcommand{\ceil}[1]{\lceil #1\rceil}

\newcommand{\rs}{\rec{\sqrt{5}}}

\newcommand{\quart}{\frac{1}{4}}

\title{Slow Fibonacci Walks}
\author{
Fan  Chung\footnote{Dept.\ of Mathematics, UCSD 
	{\tt fan@ucsd.edu} }
\and
Ron Graham\footnote{Dept. of Computer Science and Engineering, UCSD,  {\tt graham@ucsd.edu}. }
\and
Sam Spiro\footnote{Dept.\ of Mathematics, UCSD 
	{\tt sspiro@ucsd.edu} }
}
\date{\today}
\begin{document}
\maketitle

\begin{abstract}
For a positive integer $n$, we study the number of steps to reach $n$ by  a {\it Fibonacci walk} for some  starting pair $a_1$ and $a_2$  satisfying the recurrence of $a_{k+2}=a_{k+1}+a_k$. The problem of slow Fibonacci walks, first suggested by Richard Stanley, is to determine the maximum number $s(n)$ of steps for such a Fibonacci walk ending at $n$. Stanley conjectured that for most $n$, there is a slow Fibonacci walk reaching $n = a_s$   with the property that $a_{s+1}$ is the integer closest to $\phi n$ where $\phi=(1+\sqrt{5})/2$. We prove that this is true for only a positive fraction of $n$. We give explicit formulas for the choice of the starting pairs and the determination of $s(n)$ by giving a characterization theorem.  We also derive a number of density results concerning the distribution of down and up cases (that is, those $n$ with $a_{s+1}=\lfloor \phi n\rfloor$ or $\lceil \phi n \rceil$, respectively), as well as for more general ``paradoxical'' cases.

\end{abstract}
\section{Introduction}
Given two positive integers $a_1$,$a_2$, we define the {\it $(a_1,a_2)-$Fibonacci walk} to be the sequence $w_k = w_k(a_1,a_2)$
with $w_1=a_1$, $w_2=a_2$ and $w_{k+2} = w_{k+1} + w_k$ for $k \geq 1$.  In this paper, we are interested in {\it slow} Fibonacci walks. To this end, define
$s(n;a_1,a_2)$ to be the integer $s$ such that $w_s(a_1,a_2) = n$, with this value being $-\infty$ if no such $s$ exists. 
Let $s(n) = \displaystyle\max_{a_1,a_2 \geq 1} s(n;a_1,a_2)$. We will say that the pair $(a_1,a_2)$ is $n$-good if $a_1, a_2 \geq 1$ and if $s(n) = s(n;a_1,a_2)$.
If $(a_1,a_2)$ is an $n$-good pair, then we will say that its associated sequence $w_k(a_1,a_2)$ is an $n$-slow Fibonacci walk.  

For example, it is easily seen that $s(6)=4$ and that the only 6-slow Fibonacci walks are $w_k(2,2)$ and $w_k(4,1)$.  As another example, $s(1)=2$ and $w_k(a,1)$ is a 1-slow Fibonacci walk for any $a\ge 1$. We will see that this sort of behavior is unique to the case $n=1$.

Some years ago, Richard Stanley \cite{S} suggested studying the properties of slow Fibonacci walks. Part of his motivation was to create magic tricks based on
properties of Fibonacci walks (we will mention several such tricks in Section 6.1). In particular, he conjectured that for most $n$ there exists an $n$-slow Fibonacci walk $w_k$ such that $w_{s(n)+1}=N(\phi n)$, where $\phi = \frac{1+\sqrt 5}{2}$ and $N(x)$ denotes the integer closest to $x$.  For example, \[w_5(2,2)=10=N(9.70\ldots)=N(\phi\cdot 6),\] so 6 has this property.  Conversely, one can verify that $w_k(2,1)$ is the only 4-slow Fibonacci walk, that $w_4(2,1)=4$, and that
\[
	w_5(2,1)=7\ne N(6.47\ldots)=N(\phi \cdot 4),
\]
so 4 fails to have this property.  In Corollary~\ref{C-5} we will see precisely how many such $n$ (fail to) have this property.

To state our results, we first define the standard Fibonacci sequence $f_k$ recursively by
$f_1=f_2=1$ and $f_{k+2} = f_{k+1} + f_k$ for $k\geq 1$. It is well known \cite{GKP} that $f_k$ has the explicit representation
\begin{align*}
f_k = \frac{1}{\sqrt 5} \l(\phi^k - (-\phi)^{-k}\r).
\end{align*}
As usual, let $\floor{x}$ denote the floor function of $x$, and let $\ceil{x}$ denote the ceiling function of $x$. Our main result is the following characterization theorem.

\begin{thm}\label{T-Char}
	For $n\ge 2$, there exists unique integers $a=a(n),\ b=b(n)$ and $t=t(n)$ such that $n=af_t+bf_{t-1}$ with $t\ge2$ and $1\le a\le b\le f_{t}$.  Moreover, the following holds.
	\begin{itemize}
		\item $(b,a)$ is $n$-good and $s=s(n)=t+1$.  $w_{s+1}(b,a)=\floor{\phi n}$ if $t$ is even and $w_{s+1}(b,a)=\ceil{\phi n}$ if $t$ is odd.
		\item If $a\le f_{t-1}$, then $(b,a)$ is the unique $n$-good pair.  Otherwise, the only other $n$-good pair is $(b',a')=(b+f_{t},a-f_{t-1})$ and we have $w_{s+1}(b',a')=\floor{\phi n}-1$ if $t$ is even and $w_{s+1}(b',a')=\ceil{\phi n}+1$ if $t$ is odd.
	\end{itemize}
\end{thm}
We emphasize that the $n$-good pair is $(b,a)$ and not $(a,b)$ as might be expected.    With this characterization, we will be able to prove a number of results concerning slow Fibonacci walks.  For example, we can obtain a density result for the number of $n$ with two $n$-good pairs.

\begin{thm}\label{T-Den2}
	Let $T(n)=n^{-1}|\{m\le n:m\tr{ has two }m\tr{-good pairs}\}|$.  Given $n$, let $c,p$ be such that $n=\rec{\sqrt{5}}c\phi^p$ with $\rec{\sqrt{5}}\le c<\rec{\sqrt{5}}\phi$.  Then
	\[
	T(n)=\begin{cases*}
	\rec{2\sqrt{5}\phi^4c}+O(n^{-1/2}) & $p\equiv 1\mod 2$,\\ 
	\f{\sqrt{5}}{2}c+\f{1+\phi^{-5}}{2\sqrt{5}c}-1+O(n^{-1/2})& $p\equiv 0\mod 2,\ c\le\f{1+\phi^{-3}}{\sqrt{5}}$,\\ 
	1-\f{\sqrt{5}}{2}\phi^{-1}c-\f{1+\phi^{-2}}{2\sqrt{5}c} +O(n^{-1/2})& $p\equiv 0\mod 2,\ c\ge \f{1+\phi^{-3}}{\sqrt{5}}$.
	\end{cases*}
	\]
\end{thm}

In Figure \ref{f1}, we show plots comparing the actual count of $T(n)$ (= Data) versus what Theorem \ref{T-Den2} (= Theory) predicts asymptotically.

\begin{figure}[htb] 
	\centering
	\begin{subfigure}[h]{0.45\textwidth}
		\includegraphics[width=\textwidth]{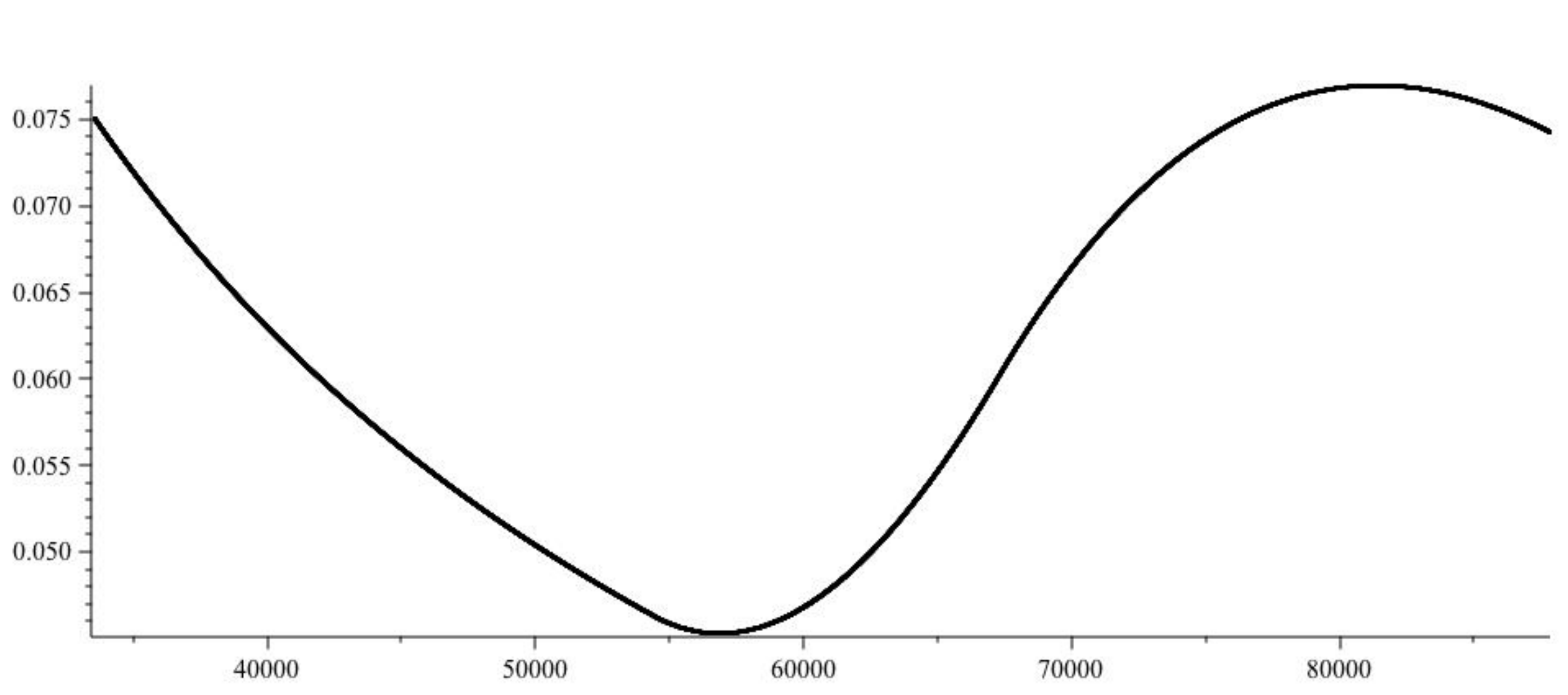}
		\caption{Data plot of $T(n).$}
	\end{subfigure}
	\begin{subfigure}[h]{0.45\textwidth}
		\includegraphics[width=\textwidth]{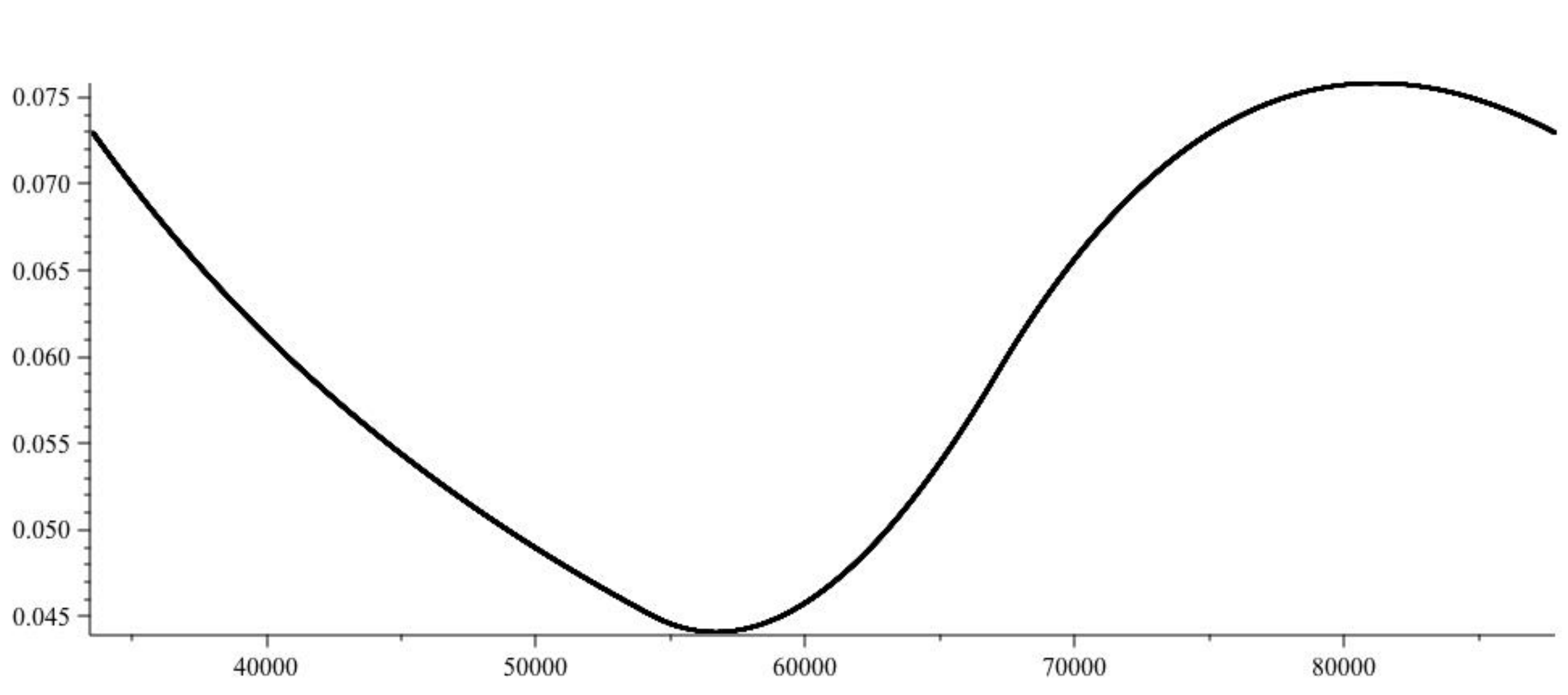}
		\caption{Theory plot of $T(n).$}
	\end{subfigure}
	\caption{Different plots of $T(n).$}
	\label{f1}
\end{figure}

We will say that $n\ge 2$ is a down-integer if $w_{s+1}=\floor{\phi n}$ for some $n$-slow Fibonacci walk, and we will say that it is an up-integer if $w_{s+1}=\ceil{\phi n}$ for some $n$-slow Fibonacci walk. We let $D=\{d_1,d_2,\ldots\}$ denote the set of down-integers written in increasing order, and similarly we define the set of up-integers $U=\{u_1,u_2,\ldots\}$.  Note that Theorem~\ref{T-Char} shows that every $n\ge 2$ belongs to precisely one of these sets.  The first few elements of these sets are listed below.

\begin{figure}[h]
	\centering
		\includegraphics[width=0.75\textwidth]{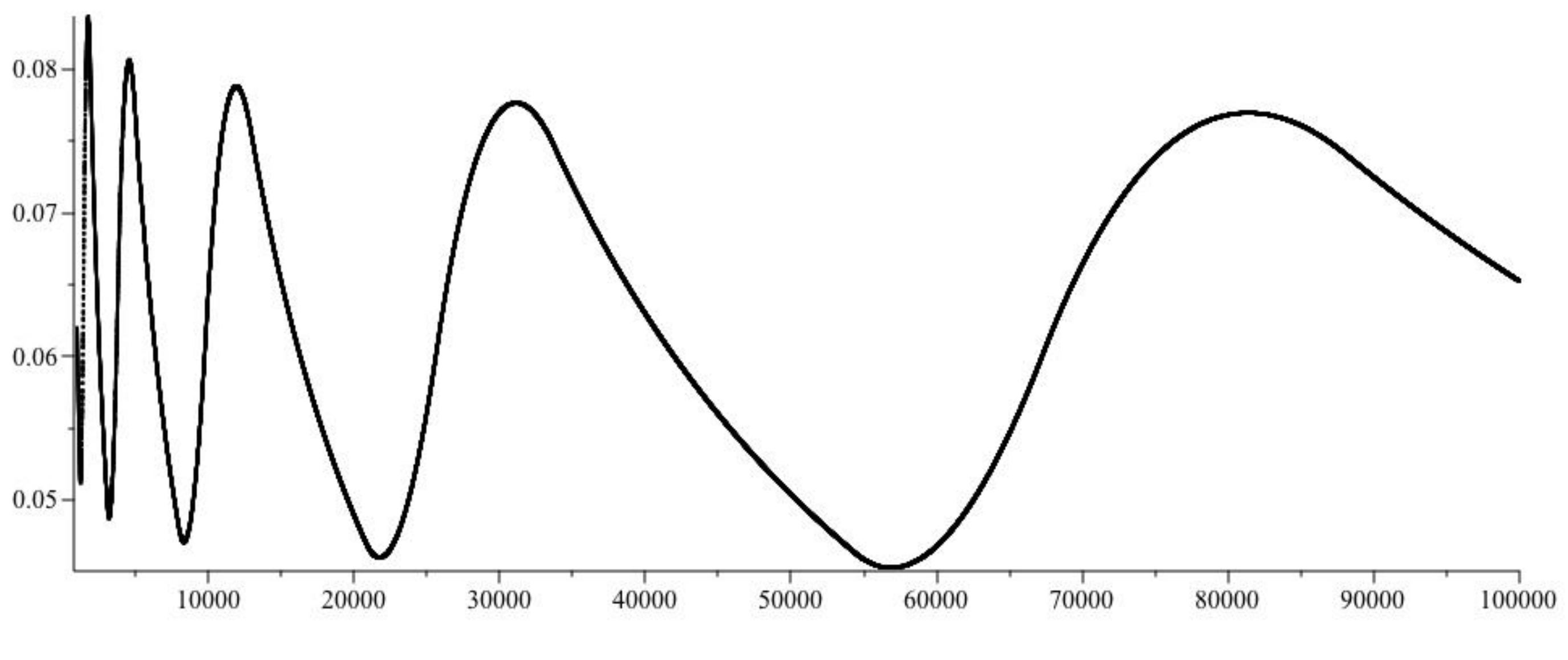}
		\caption{Data plot of $T(n)$ for $1000 \leq n \leq 100000.$}
	\label{f22}
\end{figure}

{\footnotesize
	 \begin{align}
	 D&=\{2,5, 7, 9, 10, 12, 13, 15, 18, 23, 26, 28, 31, 33, 34, 36, 38, 39, 41, 43, 44, 46, 47, 48,\ldots\}, \label{D1}\\
	 U&=\{3,4, 6, 8, 11, 14, 16, 17, 19, 20, 21, 22, 24, 25, 27, 29, 30, 32, 35, 37, 40, 42, 45, 50\ldots\}.\label{U1}
	\end{align} 
	\par}

Intuitively one might expect the densities of these sets to be roughly equal to one another.  This turns out to be correct, though as in Theorem~\ref{T-Den2} the exact densities oscillate with $n$.

\begin{thm}\label{T-DenD}
	Let $D(n)=n^{-1}|D\cap [n]|$.  Given $n$, let $c,p$ be such that $n=\rec{\sqrt{5}}c\phi^p$ with $\rec{\sqrt{5}}\le c<\rec{\sqrt{5}}\phi$.  Then 
	
	\[
	D(n)=\begin{cases*}
	1-\half c-\rec{10c}+O(n^{-1/2}) & $p\equiv 0\mod 4$,\\ 
	\rec{2\phi}c+\f{\phi}{10c}+O(n^{-1/2}) & $p\equiv 1\mod 4$,\\ 
	\half c+\rec{10c}+O(n^{-1/2}) & $p\equiv 2\mod 4$,\\ 
	1-\rec{2\phi}c-\f{\phi}{10c}+O(n^{-1/2}) & $p\equiv 3\mod 4$.
	\end{cases*}
	\]	
\end{thm}
We note that the above statement can be written  more compactly as follows:
\[
D(n)=\begin{cases*}
\f{\sqrt{5}n}{2\phi^{q+1}}+\f{\phi^{q+1}}{10\sqrt{5}n}+O(n^{-1/2}) & $\rec{5}\phi^q\le n<\rec{5}\phi^{q+2},\ q\equiv 1\mod 4$,\\ 
1-\f{\sqrt{5}n}{2\phi^{q+1}}-\f{\phi^{q+1}}{10\sqrt{5}n}+O(n^{-1/2}) & $\rec{5}\phi^q\le n<\rec{5}\phi^{q+2},\ q\equiv 3\mod 4$.
\end{cases*}
\]
In Figure \ref{f2}, we show plots comparing the actual count of $D(n)$ (= Data) versus what Theorem \ref{T-DenD} (= Theory) predicts asymptotically.

\begin{figure}[htb] 
	\centering
	\begin{subfigure}[h]{0.45\textwidth}
		\includegraphics[width=\textwidth]{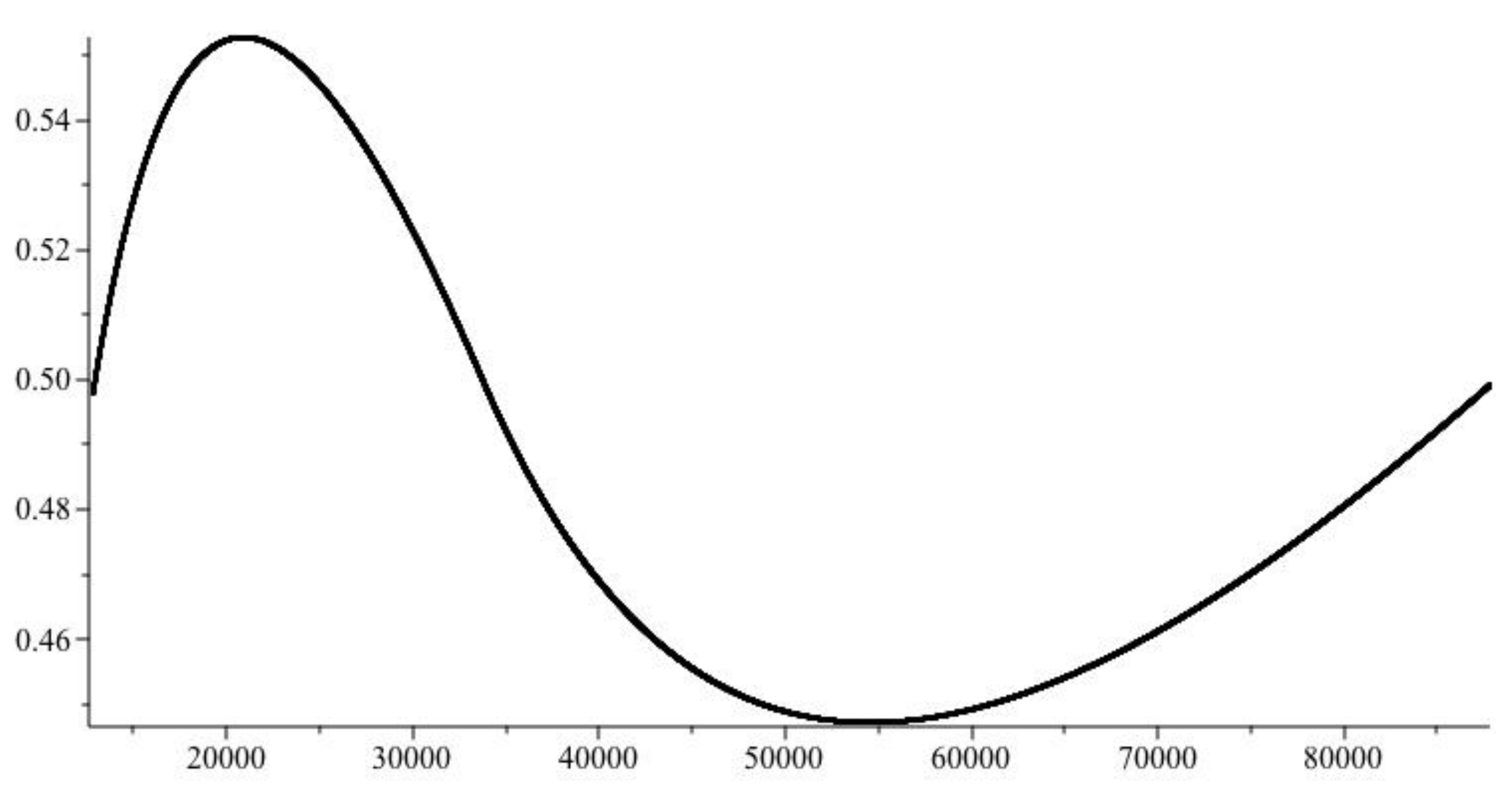}
		\caption{Data plot of $D(n).$}
	\end{subfigure}
	\begin{subfigure}[h]{0.45\textwidth}
		\includegraphics[width=\textwidth]{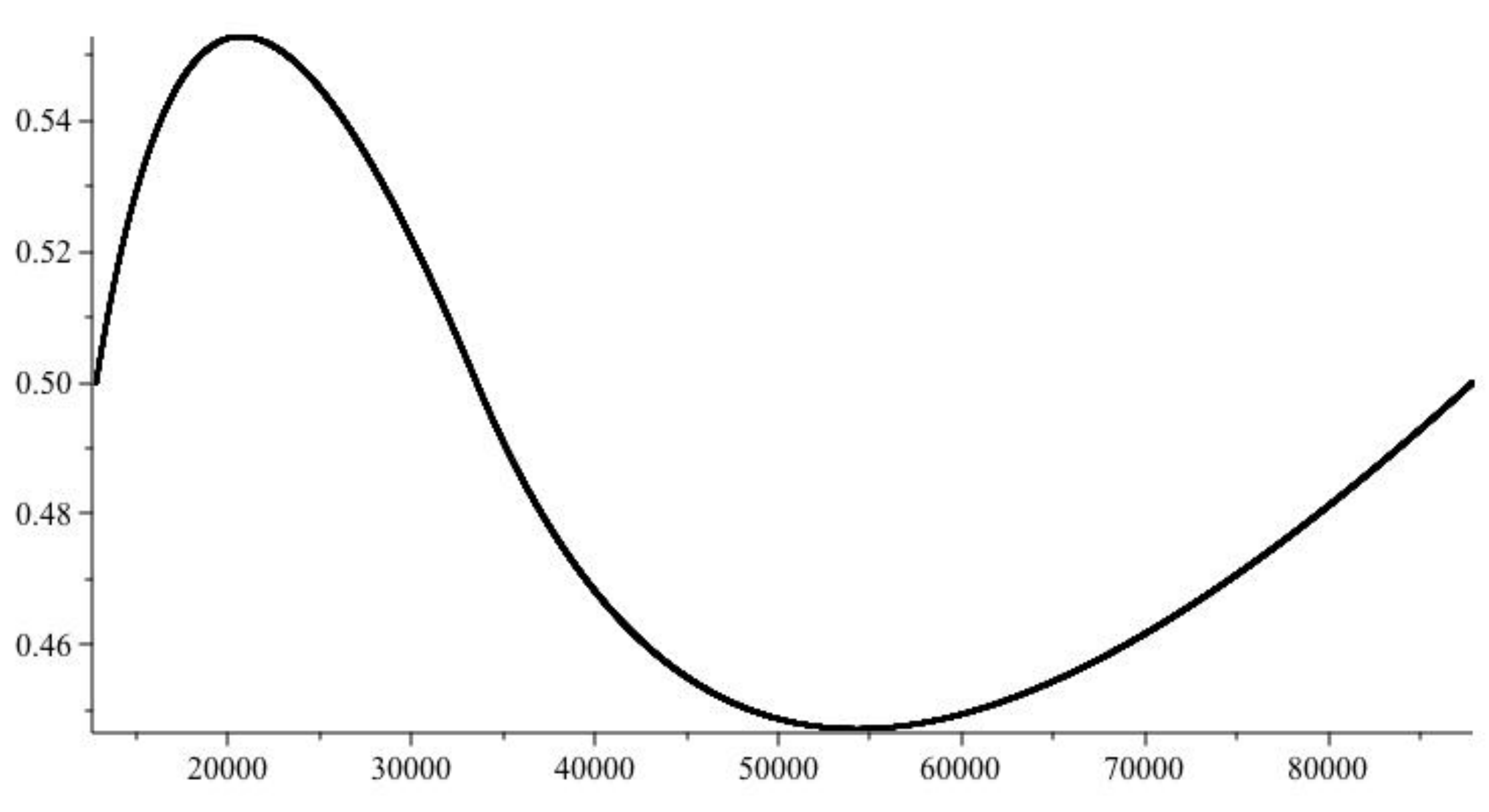}
		\caption{Theory plot of $D(n).$}
	\end{subfigure}
	\caption{Different plots of $D(n).$}
	\label{f2}
\end{figure}

Returning to Stanley's original conjecture, we define $\del_n=\phi n-\floor{\phi n}$ and similarly $\Del_n=\ceil{\phi n}-\phi n$.  Intuitively, the smaller $\del_n$ or $\Del_n$ is, the more likely it should be that $n\in D$ or $n\in U$, respectively.  To make this idea precise, we say that $n$ is $d$-paradoxical if either $\del_n<d$ and $n\in U$ or if $\Del_n<d$ and $n\in D$.

\begin{thm}\label{T-DenP}
	For $d\le \half$, let $P(n,d)=n^{-1}|\{m\le n:m\tr{ is }d\tr{-paradoxical}\}|$.  Given $n$, let $c,p$ be such that $n=\rec{\sqrt{5}}c\phi^p$ with $\rec{\sqrt{5}}\le c<\rec{\sqrt{5}}\phi$.  We have $P(n,d)=0$ if $d\le \rec{\sqrt{5}}\phi^{-1}$, and otherwise
	{\small \[
	P(n,d)=\begin{cases*}
	-\half \phi^{-1}c+d+\l( d^2-d+\rec{2\sqrt{5}}\phi^{-1}\r)c^{-1}+O(n^{-1/2}) & $p$\tr{ odd}, $c\le \phi d$,\\ 
	\f{\sqrt{5}}{2}\phi\l(d-\rec{\sqrt{5}}\phi^{-1}\r)^2c^{-1}+O(n^{-1/2}) & $p$\tr{ odd}, $c\ge \phi d$,\\ 
	-\half c+d+\l( \phi^{-1}d^2-\phi^{-1}d+\rec{2\sqrt{5}}\phi^{-2}\r)c^{-1}+O(n^{-1/2}) & $p$\tr{ even}, $c\le d$,\\
	\f{\sqrt{5}}{2}\l(d-\rec{\sqrt{5}}\phi^{-1}\r)^2c^{-1}+O(n^{-1/2}) & $p$ even, $d\le c\le 1-d$\\
	\f{1}{2}c+d-1+\l(\phi d^2-\phi d+\rec{2\sqrt{5}}\phi^2\r)c^{-1}+O(n^{-1/2}) & $p$\tr{ even}, $c\ge 1-d$.
	\end{cases*}
	\]}
\end{thm}

In particular, we get the following result when $d=\half$.
\begin{cor}\label{C-5}
	Let $P(n)$ denote the fraction of $m\le n$ such that either $N(\phi m)=\floor{\phi m}$ and $m\in U$, or $N(\phi m)=\ceil{\phi m}$ and $m\in D$.  Given $n$, let $c,p$ be such that $n=\rec{\sqrt{5}}c\phi^p$ with $\rec{\sqrt{5}}\le c<\rec{\sqrt{5}}\phi$.  Then
	\[
	P(n)=\begin{cases*}
	-\half \phi^{-1}c+\half+\l(\rec{2\sqrt{5}}\phi^{-1}-\quart \r)c^{-1}+O(n^{-1/2})  & $p$\tr{ odd},\\ 
	-\half c+\half+\l( \rec{2\sqrt{5}}\phi^{-2}-\quart\phi^{-1}\r)c^{-1}+O(n^{-1/2}) & $p$\tr{ even}, $c\le \half$,\\
	\f{1}{2}c-\half+\l( \rec{2\sqrt{5}}\phi^2-\quart\phi\r)c^{-1}+O(n^{-1/2}) & $p$\tr{ even}, $c\ge \half$.
	\end{cases*}
	\]
\end{cor}
In Figure \ref{f3}, we show plots comparing the actual count of $P(n,d)$ (= Data) versus what Theorem \ref{T-Den2} (= Theory) predicts asymptotically for various values of $d$.

\begin{figure}[htb] 
	\centering
	\begin{subfigure}[h]{0.45\textwidth}
		\includegraphics[width=\textwidth]{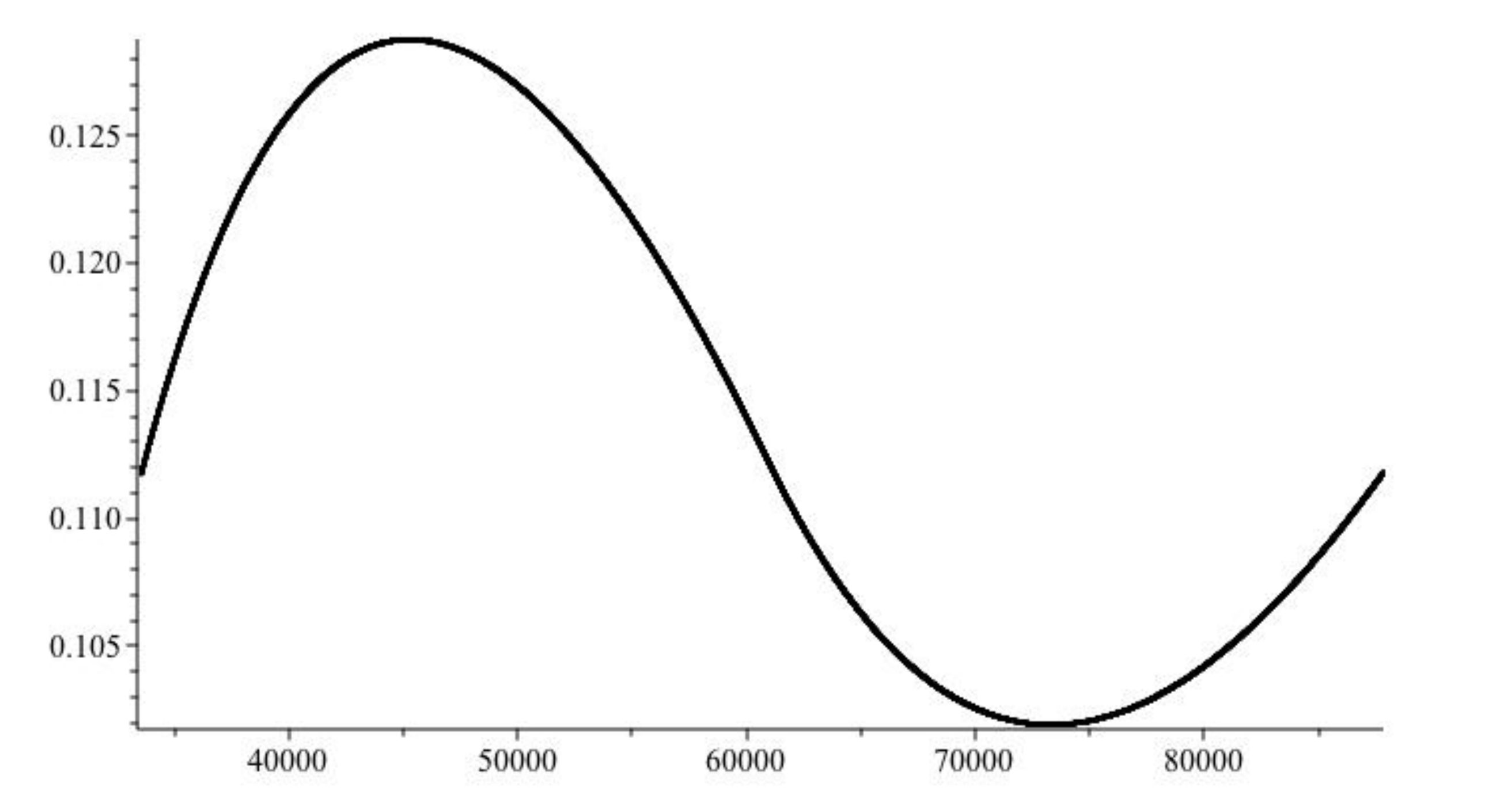}
		\caption{Data plot of $P(n,.5).$}
	\end{subfigure}
	\begin{subfigure}[h]{0.45\textwidth}
		\includegraphics[width=\textwidth]{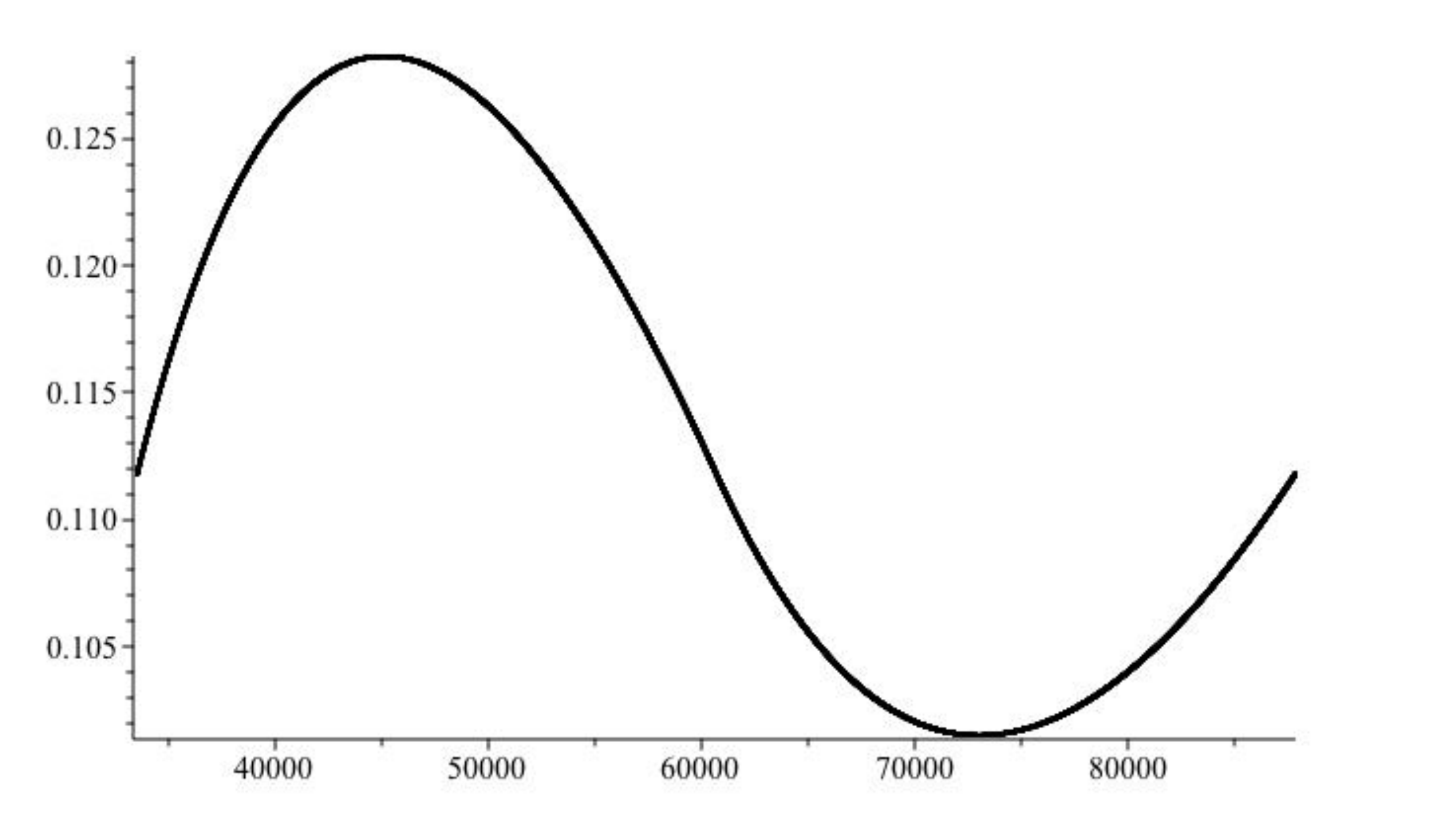}
		\caption{Theory plot of $P(n,.5).$}
	\end{subfigure}
	\begin{subfigure}[h]{0.45\textwidth}
		\includegraphics[width=\textwidth]{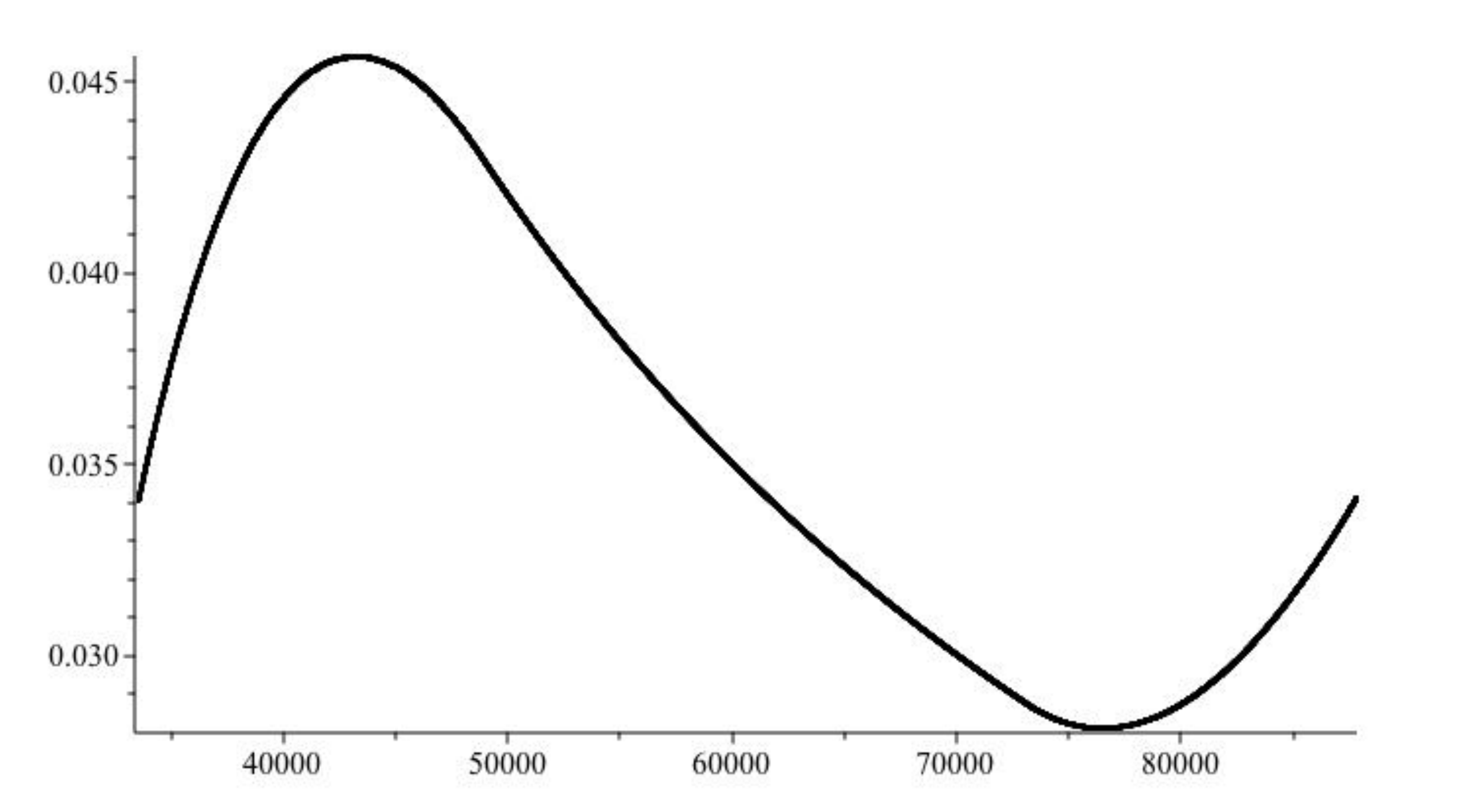}
		\caption{Data plot of $P(n,.4).$}
	\end{subfigure}
	\begin{subfigure}[h]{0.45\textwidth}
		\includegraphics[width=\textwidth]{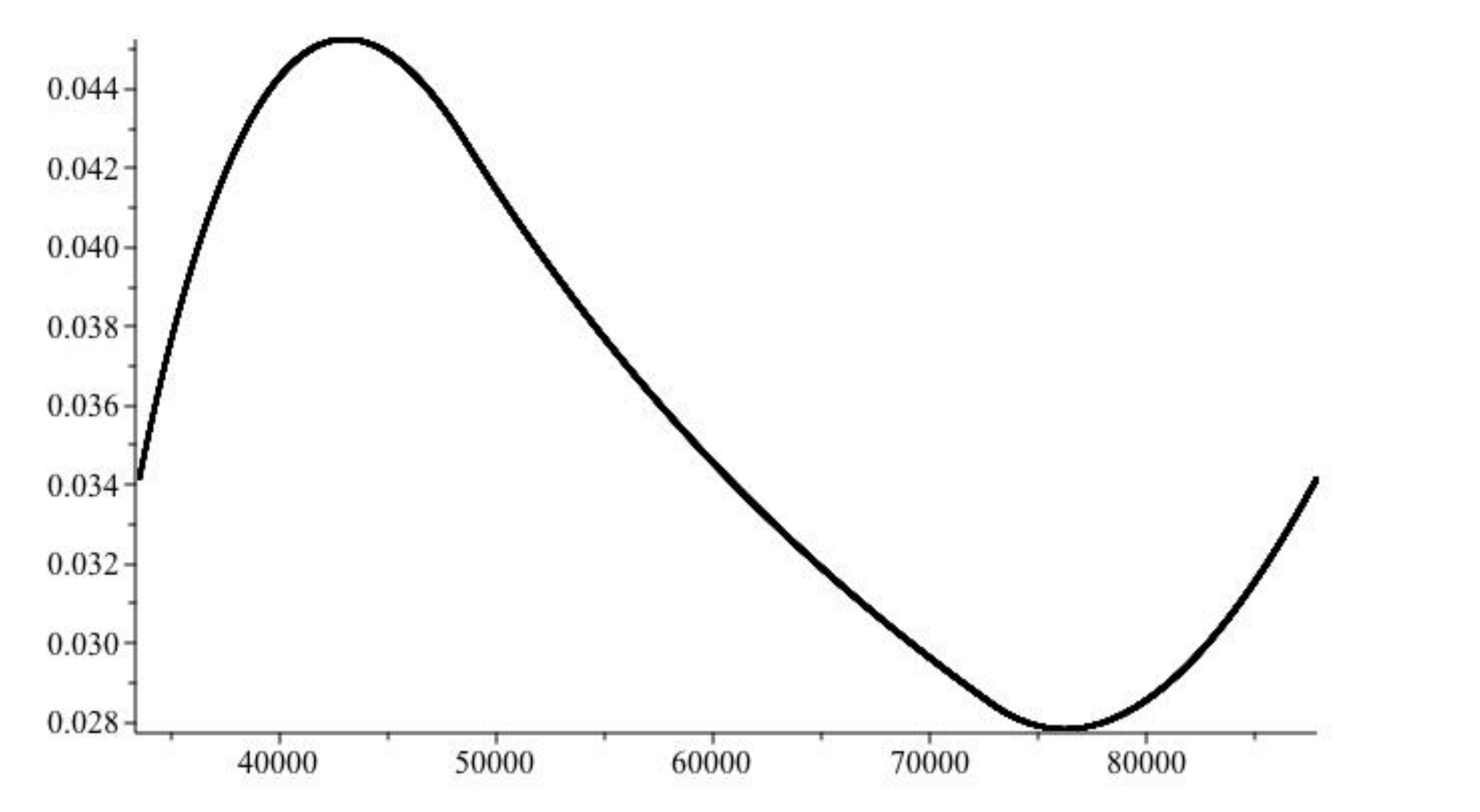}
		\caption{Theory plot of $P(n,.4).$}
	\end{subfigure}
	\begin{subfigure}[h]{0.45\textwidth}
		\includegraphics[width=\textwidth]{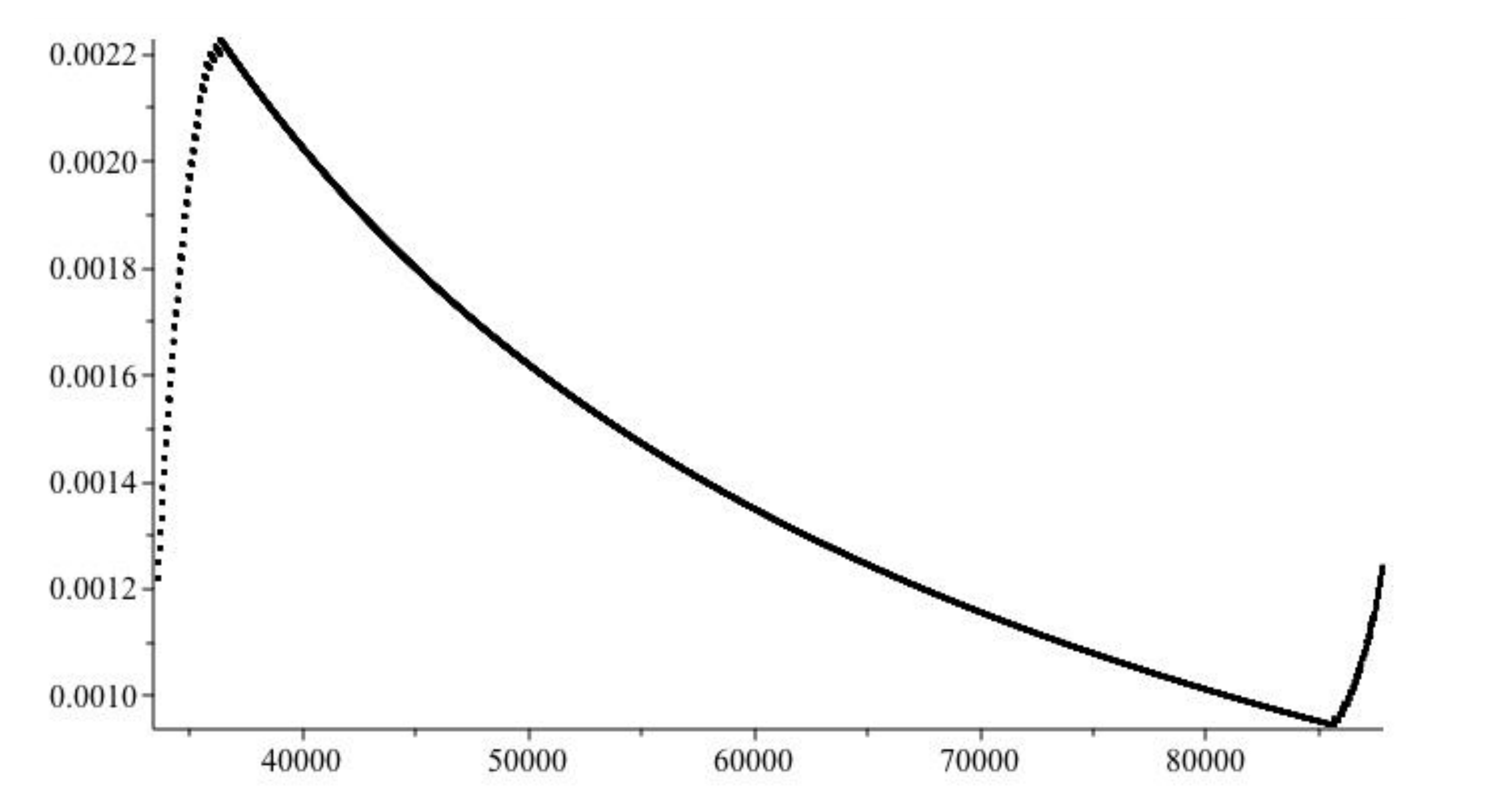}
		\caption{Data plot of $P(n,.3).$}
	\end{subfigure}
	\begin{subfigure}[h]{0.45\textwidth}
		\includegraphics[width=\textwidth]{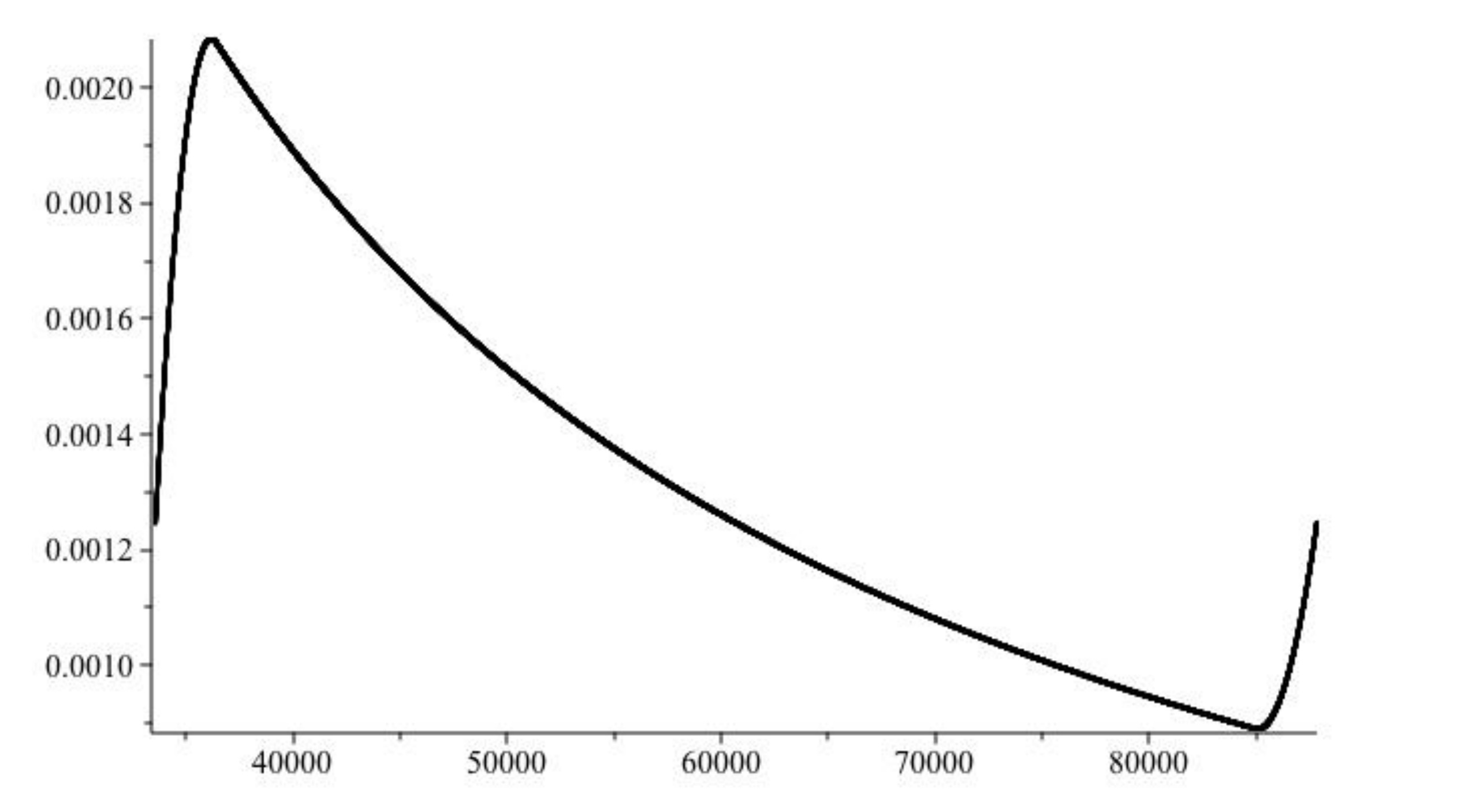}
		\caption{Theory plot of $P(n,.3).$}
	\end{subfigure}
	\caption{Different plots of $P(n,d)$ for different values of $d$.}
	\label{f3}
\end{figure}

Another natural question to ask is, for example, how large the gap size $d_{k+1}-d_{k}$ can be.  That is, how long can one go without seeing any down-integers?  From the first few terms listed in \eqref{D1}, we see that this difference can be 1, 2, 3 or 5.  Similarly one sees from \eqref{U1} that $u_{k+1}-u_k$ can also be 1, 2, 3, or 5.  Remarkably, these are the only four differences that can occur.

\begin{thm}\label{T-Gap}
	We have
	\begin{align*}
	\{d_{k+1}-d_{k}:k\ge 1\}=\{u_{k+1}-u_k:k\ge 1\}=\{1,2,3,5\}.
	\end{align*}
\end{thm}

A similar result holds for the difference sets $d_{k+2}-d_k$ and $u_{k+2}-u_k$.
\begin{thm}\label{T-Gap2}
	We have
	\begin{align*}
	\{d_{k+2}-d_{k}:k\ge 1\}=\{u_{k+2}-u_k:k\ge 1\}=\{2,3,4,5,6,8,10\}.
	\end{align*}
\end{thm}

The rest of the paper is organized as follows.  In Section~\ref{S-Char} we prove Theorem~\ref{T-Char}.  We then apply our characterization theorem to prove our first two density results in Section~\ref{S-Den}.  In Section~\ref{S-Par} we prove Theorem~\ref{T-DenP}, along with some additional results related to paradoxical $n$.  In Section~\ref{S-Gaps} we prove our gap results.  In Section~\ref{S-End} we present a different way to view slow Fibonacci walks.  We use this change in perspective to give an elegant proof that there are at most two $n$-good pairs, to construct an $O(\log n)$ algorithm for finding all $n$-slow Fibonacci walks, and to give a magic trick involving  Fibonacci walks which is detailed in Section~\ref{S-Mag}.  We end with some concluding remarks and open questions in Section~\ref{S-Con}.

\section{The Characterization Theorem}\label{S-Char}
We derive some formulas for $w_k(b,a)$.  We adopt the convention that $f_0=0$ and $f_{-1}=1$.
\begin{lem}\label{L-Form}
	Let $a,b,k\ge 1$.
	\begin{enumerate}
		\item[(a)] \begin{align*}w_k(b,a)&=a f_{k-1}+bf_{k-2}\\&=\rec{\sqrt{5}}\l(a\phi^{k-1}+b\phi^{k-2}-a(-\phi)^{-k+1}-b(-\phi)^{-k+2}\r),\end{align*}
		
		\item[(b)] \[w_{k+1}(b,a)=\phi w_k(b,a)+(-\phi)^{1-k}(a-\phi b).\]
	\end{enumerate}
\end{lem}

\begin{proof}
	The first equality of (a) follows by an easy induction argument, and the second comes from substituting in the closed formula for the Fibonacci numbers.
	
	For (b), we use (a) and the closed formula for the Fibonacci numbers to conclude
	\begin{align*}w_{k+1}(b,a)=&\rec{\sqrt{5}}\l(a\phi^{k}+b\phi^{k-1}-a(-\phi)^{-k}-b(-\phi)^{-k+1}\r)\\ 
	=&\rec{\sqrt{5}}\l(a\phi^{k}+b\phi^{k-1}+a(-\phi)^{-k+2}+b(-\phi)^{-k+3}\r)\\ &-\rec{\sqrt{5}}\l(a(-\phi)^{-k}+b(-\phi)^{-k+1}+a(-\phi)^{-k+2}+b(-\phi)^{-k+3}\r)\\ 
	=&\phi w_k(b,a)+\f{(1+\phi^2)}{\sqrt{5}}(-\phi)^{-k}(\phi b-a),\end{align*}
	
	and we get our final result by observing that $\f{1+\phi^2}{\sqrt{5}}=\phi$.
\end{proof}

We next derive some structural results for $n$-good pairs.  We recall that $f_k$ and $f_{k+1}$ are relatively prime for $k\ge 1$, as well as Cassini's identity \cite{GKP} \[f_{k-1}f_{k+1}-f_{k}^2=(-1)^k.\]
\begin{lem}\label{L-Facts}
	Assume that $(b,a)$ is $n$-good with $s(n)=s$.
	\begin{itemize}
		\item[(a)] We have $a\le b$.
		
		\item[(b)] The pair $(b',a')$ with $a',b'\ge 1$ is $n$-good if and only if $a'=a+kf_{s-2}\ge 1$ and $b'=b-k f_{s-1}\ge 1$  for some $k\in \Z$. 
		
		\item[(c)] With $k$ as above, we have $w_{s+1}(b,a)-w_{s+1}(b',a')=(-1)^sk$.
	\end{itemize}
\end{lem}
\begin{proof}
	We first note that $(b,a)$ being $n$-good together with Lemma~\ref{L-Form} implies that $n=w_{s}(b,a)=af_{s-1}+bf_{s-2}$ and that there exists no $a',b'\ge 1$ with $n=a'f_{s}+b'f_{s-1}$.
	
	For (a), having $a>b$ would imply $n=bf_s+(a-b)f_{s-1}$ with $b,a-b\ge 1$, a contradiction to the remarks made above. 
	
	For (b), note that $(b',a')$ is $n$-good if and only if it is a positive solution to the Diophantine equation $n=a'f_{s-1}+b'f_{s-2}$ defined in Lemma~\ref{L-Form}(a).  The result then follows from the fact that $f_{s-1}$ and $f_{s-2}$ are relatively prime.  
	
	For (c), we have by Lemma~\ref{L-Form}(a) that \begin{align*}w_{s+1}(b,a)-w_{s+1}(b',a')&=af_{s}+bf_{s-1}-(a+kf_{s-2})f_{s}-(b-kf_{s-1})f_{s-1}\\ &=k(f_{s-1}^2-f_{s-2}f_{s})=(-1)^s k,\end{align*} where the last step follows from Cassini's identity.
\end{proof}

We next give a lower bound for $s(n)$.
\begin{lem}\label{L-Chicken}
	Let $s$ be a fixed integer.  If $n> f_{s-1}f_{s-2}$, then $s(n)\ge s$.
\end{lem}
\begin{proof}
	By Lemma~\ref{L-Form}(a), $s(n)$ is the largest integer such that there exists some $a,b\ge 1$ with $n= af_{s(n)-1}+bf_{s(n)-2}$.  Thus it will be enough to show that there exists $a',b'\ge 0$ such that $a'f_{s-1}+b'f_{s-2}=n-f_{s-1}-f_{s-2}$.  By the Frobenius Coin Problem \cite{Syl}, this is possible provided $n-f_{s-1}-f_{s-2}>f_{s-1}f_{s-2}-f_{s-1}-f_{s-2}$, proving the result.
\end{proof}

We can now establish some strong bounds on the sizes of $a$ and $b$.
\begin{lem}\label{L-abBound}
	If $(b,a)$ is $n$-good with $s(n)=s>2$, then $a\le f_{s-1}$ and $b\le 2f_{s-1}$.
\end{lem}
\begin{proof}
	If $a> f_{s-1}$, then by Lemma~\ref{L-Facts}(a) we would have $n\ge af_{s-1}+af_{s-2}>f_{s}f_{s-1}$, which implies that $s(n)\ge s+1$ by Lemma~\ref{L-Chicken}, a contradiction.
	
	For the bound on $b$, assume for contradiction that $b=kf_{s-1}+r$ with $r<f_{s-1}$ and $k\ge 2$.  In this case another $n$-good pair is $(b',a')=(r,a+kf_{s-2})$ by Lemma~\ref{L-Facts}.  If $k>2$, then this implies that $a'>2f_{s-2}>f_{s-1}$ (since $f_{s-2}\ne 0$), which can not happen by what we have just proven, so we can assume $k=2$ and $r>0$.  In order for this pair to be $n$-good, we need $b'\le a'$ by Lemma~\ref{L-Facts}(a), and hence \[r=b'\ge a'\ge 1+2f_{s-2}\ge 1+f_{s-1}\] since $s>2$, a contradiction to how $r$ was defined, so we conclude the desired bound.
\end{proof}

We are now able to prove our main result.

\begin{proof}[Proof of Theorem~\ref{T-Char}]
	By Lemma~\ref{L-Form}(a), $s(n)=s$ implies that $(b,a)$ is $n$-good for some $a,b\ge 1$ such that $n=af_{s-1}+bf_{s-2}$.  Note that $n\ge2$ implies that $s>2$.  Thus we can assume that $b\le f_{s-1}$, as otherwise we could instead consider the $n$-good pair $(b-f_{s-1},a+f_{s-2})$, noting that by Lemma~\ref{L-abBound} we have $b-f_{s-1}\le f_{s-1}$.  Since $(b,a)$ is $n$-good, we must have $a\le b\le f_{s-1}$ by Lemma~\ref{L-Facts}(a).  Thus choosing $t(n)=s-1,\ a(n)=a$, and $b(n)=b$ shows that such integers exist.  Moreover, $(b,a)$ is $n$-good and $s(n)=t+1$ by construction.
	
	To show that these integers are unique, assume that $n=af_{t}+bf_{t-1}$ with $a,b,t$ as in the hypothesis of the theorem, and assume that $s(n)\ge t+2$.  This implies that there exists $a',b'\ge 1$ such that $n=a'f_{t+1}+b' f_{t}=(a'+b')f_{t}+a'f_{t-1}$.  Since $f_{t},f_{t-1}$ are relatively prime, this implies that there exists a $k$ such that \begin{align*}a'&=b+kf_{t},\\  b'&=a-a'-kf_{t-1}=a-b-kf_{t+1}.\end{align*}  Since $a'\ge 1$ and $b\le f_{t}$, we must have $k\ge 0$.  Similarly since $b'\ge 1$ and $a-b\le 0$, we must have $k\le -1$.  Thus no such $k$ exits and we conclude that we must have $t=s-1$, and the uniqueness of $a,b$ of the desired form follows from Lemma~\ref{L-Facts}(b).  From now on we let $a,b,t=s-1$ denote these unique integers corresponding to $n$.
	
	By Lemma~\ref{L-Facts}(b) together with the fact that $a\le f_{t}\le 2f_{t-1}$ and $b\le f_{t}$, the only pairs that could be $n$-good are $(b,a)$ and $(b',a')=(b+f_{t},a-f_{t-1})$, with the second pair being good if and only if $a-f_{t-1}\ge 1$.  Moreover, by Lemma~\ref{L-Facts}(c) we have $w_{s+1}(b',a')=w_{s+1}(b,a)-(-1)^t$, so all that remains is to prove the result concerning $w_{s+1}(b,a)$.
	
	Observe that having $w_{s+1}(b,a)=\floor{\phi n}$ is equivalent to $\phi n-w_{s+1}(b,a)$ being a positive number less than one.  Since $n=w_{s}(b,a)$ and $s=t+1$, by Lemma~\ref{L-Form}(b) this is equivalent to having
	\[
		0<(-\phi)^{-t}(\phi b-a)<1.
	\]  Using $1\le a\le b\le f_t$, we have that \[0<\phi^{-t}(\phi b-a)\le \phi^{1-t} f_t-\phi^{-t}\le \rec{\sqrt{5}}\phi+\l(\rec{\sqrt{5}}\phi^{1-2t}-\phi^{-t}\r)\le \rec{\sqrt{5}}\phi<1.\]
	Thus our relevant quantity is always less than 1 in absolute value, and it will be positive if and only if $t$ is even.  We similarly find that $w_{s+1}(b,a)=\ceil{\phi n}$ if and only if $t$ is odd, proving the result.
\end{proof}

\section{Densities}\label{S-Den}
In this section we prove our first two density results. 

\begin{proof}[Proof of Theorem~\ref{T-Den2}]
	Given $n$, let $g(t)$ denote the number of pairs $(a,b)$ such that $f_{t-1}+1\le a\le b\le f_t$ and $af_{t}+bf_{t-1}\le n$.  Equivalently, $g(t)$ is the number of pairs $(a,b)$ satisfying 
	\begin{align}
	1\le a\le b\le f_t-f_{t-1}&=f_{t-2}, \label{E-1T'}\\ 
	af_t+bf_{t-1}\le n-f_{t-1}(f_t+f_{t-1})&=n-f_{t-1}f_{t+1}. \label{E-2T'}
	\end{align}
	Define $h(t)$ to be the number of $(a,b)$ satisfying the following two conditions:
	\begin{align}
	0\le a\le b&\le \rec{\sqrt{5}}\phi^{t-2}, \label{E-1T}\\ 
	a\phi^t+b\phi^{t-1}&\le c\phi^{p}-\rec{\sqrt{5}}\phi^{2t}. \label{E-2T}
	\end{align}
	Ultimately we are interested in computing  $g(t)$.  The following claim shows that it will be enough to compute $h(t)$, whose conditions are easier to work with.  Note that by using the closed form for $f_k$, we find that \eqref{E-2T'} is equivalent to {\small\begin{equation}a\phi^t+b\phi^{t-1}\le c\phi^{p}-\rec{\sqrt{5}}\phi^{2t}+a(-\phi)^{-t}+b(-\phi)^{-t+1}+\rec{\sqrt{5}}(-1)^t(\phi^2+\phi^{-2})+\rec{\sqrt{5}}\phi^{-2t}.\label{E-2T''}\end{equation}}
	
	\begin{claim}\label{Cl-Equiv}
		$|g(t)-h(t)|=O(\phi^t)$.
	\end{claim}
	\begin{proof}
		The statement is trivially true if $t\le 3$, so assume $t\ge 4$.  Note that $f_{t-2}$ is the closest integer to $\rec{\sqrt{5}}\phi^{t-2}$, so we always have $|f_{t-2}-\rec{\sqrt{5}}\phi^{t-2}|<1$.  Thus the only $(a,b)$ that could satisfy \eqref{E-1T'} but not \eqref{E-1T} are those with $b=\ceil{\rec{\sqrt{5}}\phi^{t-2}}$ and $a\le \ceil{\rec{\sqrt{5}}\phi^{t-2}}$, and the number of such pairs is precisely $\ceil{\rec{\sqrt{5}}\phi^{t-2}}$.  Using similar logic, we find that the only pairs satisfying \eqref{E-1T} but not \eqref{E-1T'} are those with $a=0$, and there are at most $\floor{\rec{\sqrt{5}}\phi^{t-2}}$ such pairs.  If $j(t)$ counts the number of $(a,b)$ satisfying \eqref{E-1T} and \eqref{E-2T''} (which again is equivalent to \eqref{E-2T'}), then we conclude that $|g(t)-j(t)|\le \ceil{\rec{\sqrt{5}}\phi^{t-2}}$.  It remains is to show that $|j(t)-h(t)|=O(\phi^t)$.

		Observe that since any valid pair for either $j(t)$ or $h(t)$ has $a,b\le \phi^{t-1}$, the difference between the right side of \eqref{E-2T''} and the right side of \eqref{E-2T} is less than 5.  Since $\phi^t\ge 5$ for $t\ge 4$, we conclude that if $(a,b)$ satisfies \eqref{E-1T} and \eqref{E-2T} with $a\ge 2$, then $(a-1,b)$ satisfies \eqref{E-1T} and \eqref{E-2T''}.  
		
		Given $b$, let $a_b$ denote the largest $a$ such that $(a_b,b)$ is counted by $j(t)$, with $a_b=-1$ if no such value exists.  Observe that $(a,b)$ is counted by $j(t)$ for all $0\le a\le a_b$.  Now let $b\le \rec{\sqrt{5}}\phi^{t-2}$ be fixed.  If $(a,b)$ is counted by $h(t)$ but not $j(t)$, then either $a=0$ or $(a-1,b)$ is counted by $j(t)$, so $a-1\le a_b$, and hence $a=a_b+1$ since $(a,b)$ was not counted by $j(t)$.  Thus for each of the at most $\floor{\rec{\sqrt{5}}\phi^{t-2}}$ fixed values that $b$ can take on, the only pair that could be counted by $h(t)$ but not $j(t)$ is $(a_b+1,b)$.  The same reasoning shows that there are at most this many pairs counted by $j(t)$ but not $h(t)$.  We conclude that $|j(t)-h(t)|\le \floor{\rec{\sqrt{5}}\phi^{t-2}}$ and the desired result follows.
	\end{proof}
	We now wish to estimate $h(t)$ for various $t$.  We first observe that \eqref{E-2T} implies that $h(t)=0$ whenever $2t\ge p+1$, and one can similarly see from \eqref{E-2T''} that $g(t)=0$ whenever $2t\ge p+2$ and $p$ is sufficiently large.  We note that \eqref{E-1T} implies \[a\phi^t+b\phi^{t-1}+\rec{\sqrt{5}}\phi^{2t}\le \rec{\sqrt{5}}\phi^{2t-1}+\rec{\sqrt{5}}\phi^{2t}=\rec{\sqrt{5}}\phi^{2t+1}\le c \phi^{2t+1},\] and hence \eqref{E-2T} will always be satisfied when $2t\le p-1$ and $1\le a\le b \le \rec{\sqrt{5}}\phi^{t-2}$, so we conclude that $h(t)=\rec{10}\phi^{2t-4}+O(\phi^t)$ in this case.
	
	It remains to deal with the case $2t=p$.  Define $d:=c-\rec{\sqrt{5}}$.  In this setting, by considering the extremal value $a=0$, we find that \eqref{E-1T} and \eqref{E-2T} are equivalent to
	\begin{align}
	0\le b&\le\min\l\{\rec{\sqrt{5}}\phi^{p/2-2}, d\phi^{p/2+1}\r\}, \label{E-bT}\\ 
	0\le a&\le \min\l\{b,d\phi^{p/2}-\phi^{-1}b\r\}\label{E-aT}.
	\end{align}
	We first consider $b\le d\phi^{p/2-1}$, in which case \eqref{E-aT} reduces to $a\le b$.  When $b$ is this small, \eqref{E-bT} is always satisfied since $d\le\rec{\sqrt{5}}(\phi-1)=\rec{\sqrt{5}}\phi^{-1}$.  Thus any $0\le a\le b\le d\phi^{p/2-1}$ satisfies both of these equations, giving a count of $\half d^2\phi^{p-2}+O(\phi^{p/2})$.
	
	We now consider pairs with $b\ge d\phi^{p/2-1}$, noting that in this range \eqref{E-aT} reduces to $a\le d\phi^{p/2}-\phi^{-1}b$.  If $d\le \rec{\sqrt{5}}\phi^{-3}$, then \eqref{E-bT} reduces to $b\le  d \phi^{p/2+1}$, in which case the count becomes \begin{align*}\sum_{b=d\phi^{p/2-1}}^{d\phi^{p/2+1}} d\phi^{p/2}-\phi^{-1}b&=(\phi-\phi^{-1})d^2\phi^p-\half\phi^{-1}(\phi^2-\phi^{-2}) d^2\phi^p+O(\phi^{p/2})\\ &=d^2\phi^p-\half\phi^{-1}(\phi^2-\phi^{-2}) d^2\phi^p+O(\phi^{p/2}),\end{align*} where we used $\phi-\phi^{-1}=1$ and that $\sum_{k=x}^yz-k=(y-x+1)z-{y+1\choose 2}+{x\choose 2}$.  Technically, we should be taking floors and ceilings of the terms of the above sum, as well as on the bounds that $b$ ranges through in the sum.  However, this miscalculation gets absorbed into the $O(\phi^{p/2})$ error term, so this will not affect our final result.  Adding $\half d^2\phi^{p-2}+O(\phi^{p/2})$ to this and using that $1-\half \phi^{-1}(\phi^2-\phi^{-2})+\half \phi^{-2}=\half$ gives \[h(p/2)=\half d^2\phi^p+O(\phi^{p/2})\tr{ when }d\le \rec{\sqrt{5}}\phi^{-3}.\]
	
	Now if $d\ge \rec{\sqrt{5}}\phi^{-3}$, then \eqref{E-bT} reduces to $b\le \rec{\sqrt{5}}\phi^{p/2-1}$, in which case the contribution from $b\ge d\phi^{p/2-1}$ becomes \[\sum_{b=d\phi^{p/2-1}}^{\rec{\sqrt{5}}\phi^{p/2-2}} d\phi^{p/2}-\phi^{-1}b=\phi^{-1}\l(\rec{\sqrt{5}}\phi^{-1}-d\r)d\phi^p-\half \phi^{-3}\l(\rec{5}\phi^{-2}-d^2\r)\phi^p.\] 
	
	We note that implicitly here we use that $d\le \rec{\sqrt{5}}\phi^{-1}$, as otherwise taking this sum would give us a negative number.  Adding $\half d^2\phi^{-2}\phi^p$ and using $-\phi^{-1}+\half\phi^{-2}+\half \phi^{-3}=-\half\phi^{-1}$ gives
	\[
	h(p/2)= \l(-\half \phi^{-1}d^2+\rec{\sqrt{5}}\phi^{-2}d-\rec{10}\phi^{-5}\r)\phi^p+O(\phi^{p/2})\tr{ when }d\ge \rec{\sqrt{5}}\phi^{-3}.
	\]
	
	By Theorem~\ref{T-Char}, we have $T(n)=\rec{n}\sum g(t)$, where we note that the uniqueness of $a,b,t$ ensures that we do not count some $m\le n$ twice in this sum.  Since $g(t)=0$ for $2t\ge p+2$, we can restrict our sum to the range $2t\le p+1$.  Define $H(n)=\rec{n}\sum h(t)$.  Since $h(t)=0$ for $2t\ge p+1$, we have by Claim~\ref{Cl-Equiv} that \[|T(n)-H(n)|\le O(\phi^{-p})\cdot\sum_{1\le t\le p/2+1} O(\phi^t)=O(\phi^{-p/2}).\]  Thus bounding $H(n)$ is equivalent to bounding $T(n)$ up to an $O(\phi^{-p/2})$ error term.
	
	If $p$ is odd then \begin{align*}H(n)&=\rec{n}\sum_{t\le (p-1)/2}h(t)=\rec{10\phi^4n}\sum_{t\le (p-1)/2} \phi^{2t}+O(\phi^t)\\ &=\rec{2\sqrt{5}\phi^4 c}\phi^{-p}\cdot \f{\phi^{p+1}}{\phi^2-1}+O(\phi^{-p/2})=\rec{2\sqrt{5}\phi^4 c}+O(\phi^{-p/2}),\end{align*} where we used that $\phi^2-1=\phi$. 
	
	If $p$ is even we note that \[\sum_{t\ne p/2} h(t)=\rec{10} \phi^{-4}\sum_{t\le p/2-1}\phi^{2t}+O(\phi^t)=\rec{10}\phi^{-5}\phi^p+O(\phi^{p/2}).\]  Using this together with $d^2=c^2-\f{2}{\sqrt{5}}c+\rec{5}$, we find for $p$ even and $d\le \rec{\sqrt{5}}\phi^{-3}$ that
	\[
	H(n)=\f{\sqrt{5}}{2}c-1+\f{1+\phi^{-5}}{2\sqrt{5}c}+O(\phi^{-p/2}),
	\]
	and similarly for $d\ge \rec{\sqrt{5}}\phi^{-3}$ we get
	\begin{align*}
	H(n)&=-\f{\sqrt{5}}{2}\phi^{-1}c+(\phi^{-1}+\phi^{-2})+\rec{\sqrt{5}c}\l(-\half\phi^{-1}-\phi^{-2}\r)\\&=1-\f{\sqrt{5}}{2}\phi^{-1}c-\f{1+\phi^{-2}}{2\sqrt{5}c}+O(\phi^{-p/2}).
	\end{align*}
	 The final result follows since $\phi^{-p/2}=O(n^{-1/2})$.
\end{proof}

We prove our second density result by using similar techniques.

\begin{proof}[Proof of Theorem~\ref{T-DenD}]
	Define $g(t)$ to be the number of pairs $(a,b)$ which have $1\le a\le b\le f_t$ and $af_t+bf_{t-1}\le n$, and define $h(t)$ to be the number of $(a,b)$ with 
	\begin{align}
	0\le a\le b\le \rec{\sqrt{5}}\phi^t, \nonumber\\ 
	a\phi^t+b\phi^{t-1}\le c\phi^p. \nonumber
	\end{align}
	By essentially the same proof as in Claim~\ref{Cl-Equiv}, we can show that $|g(t)-h(t)|=O(\phi^t)$.  We note that these two conditions are equivalent to
	\begin{align}
	0\le b\le\min\l\{\rec{\sqrt{5}}\phi^{t}, c\phi^{p-t+1}\r\}, \label{E-1G}\\ 
	0\le a\le \min\l\{b,c\phi^{p-t}-\phi^{-1}b\r\}\label{E-2G}.
	\end{align}
	
	If $t\le (p-1)/2$ then \eqref{E-1G} reduces to $b\le \rec{\sqrt{5}}\phi^t$ and \eqref{E-2G} reduces to $a\le b$.  We conclude for $2t+1\le p$ that any pair $(a,b)$ with $0\le a\le b\le \rec{\sqrt{5}}\phi^t$ is counted by $h(t)$, and hence \[h(t)=\rec{10}\phi^{2t}+O(\phi^t)\tr{ for any } t\le(p-1)/2.\]
	
	Now assume $2t-2\ge p$.  In this case \eqref{E-1G} reduces to $b\le c\phi^{p-t+1}$.  When $b< c\phi^{p-t-1}$, \eqref{E-2G} reduces to $0\le a\le b$, so in this case the number of valid choices we have is $\half c^2\phi^{2p-2t-2}+O(\phi^{p-t})$.  The number of choices for the case $b\ge c\phi^{p-t-1}$ is
	{\small\begin{align*}
	\sum_{b=c\phi^{p-t-1}}^{c\phi^{p-t+1}} c\phi^{p-t}-\phi^{-1}b&=(\phi-\phi^{-1})c^2\phi^{2p-2t}-\half\phi^{-1}(\phi^2-\phi^{-2})c^2\phi^{2p-2t}+O(\phi^{p-t})\\ 
	&=\l(1-\half \phi^{-1}(\phi^2-\phi^{-2})\r)c^2\phi^{2p-2t}+O(\phi^{p-t}),
	\end{align*}}
	where we used that $\phi-\phi^{-1}=1$.  Adding $\half c^2\phi^{2p-2t-2}+O(\phi^{p-t})$ to this quantity, and using that $1-\half \phi^{-1}(\phi^2-\phi^{-2})+\half \phi^{-2}=\half$, we find 
	\[h(t)=\half c^2 \phi^{2p-2t}+O(\phi^{p-t})\tr{ for any }t\ge (p+2)/2.\]
	
	It remains to deal with the cases $2t=p$ and $2t=p+1$.  First consider $2t=p$.  Note that in this range we have $c\phi^{p-p/2+1}>\rec{\sqrt{5}}\phi^{p/2}$, so \eqref{E-1G} reduces to $b\le \rec{\sqrt{5}}\phi^{p/2}$.  If $b\le c\phi^{p/2-1}$ then \eqref{E-2G} reduces to $0\le a\le b$ and the count will be $\half c^2\phi^{p-2}+O(\phi^{p/2})$.  For $b>c\phi^{p/2-1}$ the count will be \[\sum_{b=c\phi^{p/2-1}}^{\rec{\sqrt{5}}\phi^{p/2}}c\phi^{p/2}-\phi^{-1}b=\l(\rec{\sqrt{5}}c-\phi^{-1}c^2\r)\phi^p-\rec{10}\phi^{-1}\phi^p+\half \phi^{-3}c^2\phi^p+O(\phi^{p/2}).\]  
	Adding $\half c^2\phi^{p-2}+O(\phi^{p/2})$ to this and using $-\phi^{-1}+\half\phi^{-2}+\half \phi^{-3}=-\half\phi^{-1}$ gives
	\[
	h(p/2)=\l(-\half \phi^{-1}c^2+\rec{\sqrt{5}}c-\rec{10}\phi^{-1}\r)\phi^p+O(\phi^{p/2}).
	\]
	
	Finally, consider the case $2t=p+1$.  Again \eqref{E-1G} reduces to $b\le \rec{\sqrt{5}}\phi^{(p+1)/2}$.  If $b\le c\phi^{(p-3)/2}$ then \eqref{E-2G} reduces to $a\le b$, and we get a count of $\half c^2\phi^{-3}\phi^{p}+O(\phi^{p/2})$ from this.  Otherwise the count will be \[\sum_{b=c\phi^{(p-3)/2}}^{\rec{\sqrt{5}}\phi^{(p+1)/2}}c\phi^{(p-1)/2}-\phi^{-1}b=\l(\rec{\sqrt{5}}c-\phi^{-2}c^2\r)\phi^p-\rec{10}\phi^p+\half c^2\phi^{-4}\phi^p+O(\phi^{p/2}),\] so in total,
	\[
	h((p+1)/2)=\l(-\half \phi^{-2}c^2+\rec{\sqrt{5}}c-\rec{10}\r)\phi^p+O(\phi^{p/2}).
	\]

	By Theorem~\ref{T-Char}, we have $D(n)=\rec{n}\sum g(2t)$.  Let \[H_1=\sum_{\substack{t\le (p-1)/2\\ t\tr{ even}}} h(t),\ H_2=\sum_{\substack{t=p/2,(p+1)/2\\ t\tr{ even}}} h(t),\ H_3=\sum_{\substack{t\ge (p+2)/2\\ t\tr{ even}}} h(t).\] As in the proof of Theorem~\ref{T-Den2}, we find $D(n)=\rec{n}(H_1+H_2+H_3)+O(\phi^{-p/2})$.   Thus it will be enough to determine $H_1$, $H_2$, and $H_3$.  These values will depend on the value of $p\mod 4$. 
	
	First assume $p\equiv 0\mod 4$.  In this case we have 
	\[H_1=\sum_{ k\le p/4-1} h(2k)=\sum_{k\le p/4-1} \rec{10}(\phi^4)^k+O(\phi^{2k})=\f{\phi^{p}}{10(\phi^4-1)}+O(\phi^{p/2}),\] 
	\[H_2=h(p/2)=\l(-\half \phi^{-1}c^2+\rec{\sqrt{5}}c-\rec{10}\phi^{-1}\r)\phi^p+O(\phi^{p/2}),\]
	\[H_3=\sum_{k\ge p/4+1} h(2k)=\half c^2\phi^{2p}\sum_{k\ge p/4+1}(\phi^4)^{-k}+O(\phi^{p-2k})=\f{c^2\phi^p}{2(\phi^4-1)}+O(\phi^{p/2}).\]
	Dividing $H_1+H_2+H_3$ by $\rec{\sqrt{5}}c\phi^p$ gives in this case
	\begin{align*}
	D(n)&=\sqrt{5}c\l(-\half\phi^{-1}+\rec{2(\phi^{4}-1)}\r)+1+\rec{2\sqrt{5}c}\l(\rec{\phi^4-1}-\phi^{-1}\r)+O(\phi^{-p/2})\\&=1-\half c-\rec{10c}+O(\phi^{-p/2}).
	\end{align*}
	
	Now consider $p\equiv 2\mod 4$.  In this case $H_2=0$, the $H_1$ sum is now over $k\le (p-2)/4$, and the $H_3$ sum is now over $k\ge (p+2)/4$.  Effectively all this does compared to the previous case is scale $H_1$ and $H_2$ by $\phi^2$ and ignores the $H_2$ term.  Thus we have
	\[
	D(n)=\f{\sqrt{5}\phi^2c}{2(\phi^4-1)}+\f{\phi^2}{2\sqrt{5}(\phi^4-1)c}+O(\phi^{-p/2})=\half c+\rec{10 c}+O(\phi^{-p/2}).
	\]
	
	Now consider $p\equiv 3\mod 4$.  Here we have
	\[
	H_1=\sum_{k\le (p-3)/4} h(2k)=\f{\phi}{10(\phi^4-1)}\phi^p+O(\phi^{-p/2}),
	\]
	\[
	H_2=\l(-\half \phi^{-2}c^2+\rec{\sqrt{5}}c-\rec{10}\r)\phi^p+O(\phi^{-p/2}),
	\]
	\[
	H_3=\sum_{k\ge (p+5)/4} h(2k)=\f{c^2\phi^p}{2\phi(\phi^4-1)}+O(\phi^{-p/2}).
	\]
	
	We thus have
	{\small \begin{align*}
	D(n)&=\sqrt{5}c\phi^{-1}\l(-\half \phi^{-1}+\rec{2(\phi^4-1)}\r)+1+\f{\phi}{2\sqrt{5}c}\l(\f{1}{\phi^4-\phi^{-1}}-1\r)+O(\phi^{-p/2})\\&=1-\rec{2\phi}c-\f{\phi}{10c}+O(\phi^{-p/2}).
	\end{align*}}
	
	Finally the case $p\equiv 1\mod 4$ compared to the previous case has $H_1,H_3$ scaled by $\phi^2$ and $H_2=0$, so we end up with
	\[
	D(n)=\f{\sqrt{5} \phi c}{2(\phi^4-1)}+\f{\phi^3}{2\sqrt{5}(\phi^4-1)}+O(\phi^{-p/2})=\f{1}{2\phi }c+\f{\phi}{10c}+O(\phi^{-p/2}).
	\]
\end{proof}

\section{Paradoxical $n$}\label{S-Par}
Our key lemma for dealing with $d$-paradoxical $n$ will be the following.  Recall that $\del_n=\phi n-\floor{\phi n}$ and $\Del_n=\ceil{\phi n}-\phi n$.

\begin{lem}\label{L-abDel}
	Given $n$, let $a,b,t$ be as in Theorem~\ref{T-Char}.  If $t$ is even then
	\[
	\phi^{-t}(\phi b-a)=\del_n.
	\]
	
	If $t$ is odd then
	\[
	\phi^{-t}(\phi b-a)=\Del_n.
	\]
	
\end{lem}
\begin{proof}
	Assume $t$ is even.  Then Theorem~\ref{T-Char} implies that $w_{t+2}(b,a)=\floor{\phi n}$.  By Lemma~\ref{L-Form}(b) we have
	\[
	\phi^{-t}(\phi b-a)=\phi n-\floor{\phi n}=\del_n.
	\]
	The proof for $t$ odd is essentially the same.
\end{proof}

With this we can show that $n$ will not be paradoxical whenever $\del_n$ or $\Del_n$ is below a certain threshold.

\begin{prop}\label{P-Thresh}
	There exist no $n$ which are $\rec{\sqrt{5}}\phi^{-1}$-paradoxical.  Equivalently, if $\del_n< \rec{\sqrt{5}}\phi^{-1}$, then $n\in D$, and if $\Del_n< \rec{\sqrt{5}}\phi^{-1}$, then $n\in U$.
\end{prop}
\begin{proof}
	Let $n$ be such that $n\in U$ and let $a,b,t$ be as in Theorem~\ref{T-Char}.  $n\in U$ implies that $t$ is odd, and this together with the bounds $1\le a\le b\le f_t$ and Lemma~\ref{L-abDel} implies that  \[\Del_n=\phi^{-t}(\phi b-a)\le \rec{\sqrt{5}}\phi+\rec{\sqrt{5}}\phi^{-t}(\phi^{1-t}-1)\le \rec{\sqrt{5}}\phi.\]
	We conclude that $\del_n=1-\Del_n\ge 1-\rec{\sqrt{5}}\phi=\rec{\sqrt{5}}\phi^{-1}$ for all $n\in U$ as desired.
\end{proof} 

We note that Theorem~\ref{T-DenP} shows that this threshold is sharp.  More concretely, one can consider $n=f_t+f_tf_{t-1}$, which one can argue will be $(\rec{\sqrt{5}}\phi^{-1}+\ep)$-paradoxical for any $\ep>0$ when $t$ is sufficiently large.  With this proposition we can prove our density result for $d$-paradoxical $n$.

\begin{proof}
	If $d\le \rec{\sqrt{5}}\phi^{-1}$ then the result follows from Proposition~\ref{P-Thresh}, so assume this is not the case.  For notational convenience we define $r=1-d$, noting that by assumption $\half \le r\le 1-\rec{\sqrt{5}}\phi^{-1}=\rec{\sqrt{5}}\phi$.  Note that a given $n\in D$ will be $(1-r)$-paradoxical if and only if $1-r>\Del_n=1-\del_n$, and by Lemma~\ref{L-abDel} this is equivalent to having $\phi^{-t(n)}(\phi b(n)-a(n))>r$, and this same bound continues to holds if $n\in U$. With this in mind, we let $g(t)$ denote the number of pairs $(a,b)$ satisfying $1\le a\le b\le f_t,$  $af_t+bf_{t-1}\le n$, and $\phi b-a>r\phi^t$.  We define $h(t)$ to denote the number of pairs $(a,b)$ satisfying 
	\begin{align*}
	0\le a\le b&\le \rec{\sqrt{5}}\phi^t,\\ 
	a\phi^t+b\phi^{t-1}&\le c\phi^p,\\ 
	\phi b-a&\ge r\phi^t.
	\end{align*}
	As in our previous density arguments we find $|g(t)-h(t)|=O(\phi^t)$.  By considering the extremal value $a=0$ in the above inequalities, we find that $h(t)$ equivalently counts the number of pairs satisfying
	\begin{align}
	r\phi^{t-1}\le b&\le \min\l\{\rec{\sqrt{5}}\phi^t,c\phi^{p-t+1}\r\}, \label{D-Pb}\\ 
	0\le a&\le \min\{b,c\phi^{p-t}-\phi^{-1}b,\phi b-r\phi^t\} \label{D-Pa}.
	\end{align}
	Note that $b\le \phi b-r\phi^t$ is equivalent to having $b\ge r\phi^{t+1}>\rec{\sqrt{5}}\phi^t$ since $r\ge\half$, so \eqref{D-Pa} never reduces to $a\le b$ and we can ignore this case.
	
	First consider $2t\le p-1$.  In this case, as we argued in the proof of Theorem~\ref{T-DenD}, none of the bounds involving $c$ occur in \eqref{D-Pb} and \eqref{D-Pa}.  Further, note that $r\phi^{t-1}\le \rec{\sqrt{5}}\phi^t$ by our bounds on $r$, so there exist $b$ in the range of \eqref{D-Pb}.  We conclude that our count in this range will be
	\begin{align*}
		\sum_{b=r\phi^{t-1}}^{\rec{\sqrt{5}}\phi^t} \phi b-r\phi^t= \l(\rec{10}-\half\phi^{-2}r^2\r)\phi^{2t+1}-\l(\rec{\sqrt{5}}-\phi^{-1}r\r)r\phi^{2t}+O(\phi^t),
	\end{align*}
	and hence
	\begin{equation}
		h(t)=\l(\half \phi^{-1}r^2-\rec{\sqrt{5}}r+\rec{10}\phi\r)\phi^{2t}+O(\phi^t)\tr{ for }t\le (p-1)/2. \label{D-Pgen}
	\end{equation}
	
	Next consider $2t\ge p+3$.  In this case we have $c\phi^{p-2t+2}< r$ since $r\ge \half $ and $c\le \rec{\sqrt{5}}\phi$, and hence $c\phi^{p-t+1}< r\phi^{t-1}$.  We conclude that no $b$ satisfies \eqref{D-Pb}, and the same logic holds if $2t=p+2$ and $c<r$.  In these cases we have $h(t)=0$, so it only remains to deal with the cases $t=p/2,\ (p+1)/2$, and $p/2+1$ with $c\ge r$.  To this end we define $c_{2t-p}=c\phi^{p-t}$ and $r_{2t-p}=r\phi^t$.  We will freely switch between using this notation and writing these values out explicitly, depending on whether we are prioritizing using less space or being clearer with our bounds.
	
	First consider $t=p/2+1$ with $c\ge r$.  In this case (as we showed in the proof of Theorem~\ref{T-DenD}), the upper bound of \eqref{D-Pb} reduces to $c\phi^{p/2}$.  We have $\phi b-r\phi^{p/2+1}\le c\phi^{p/2-1}-\phi^{-1}b$ if and only if $b\le \rec{\sqrt{5}}(c\phi^{p/2-1}+r\phi^{p/2+1})=\rec{\sqrt{5}}(c_2+r_2)$, where we used $\phi+\phi^{-1}=\sqrt{5}$. Because $c\ge r$, we have \[r\phi^{p/2}=\phi^{-1}r_2\le \rec{\sqrt{5}}(c_2+r_2) \le \phi c_2=c\phi^{p/2},\] and hence there will exist $\rec{\sqrt{5}}(c_2+r_2)-\phi^{-1}r_2$ values of $b$ satisfying both \eqref{D-Pb} and $b\le \rec{\sqrt{5}}(c_2+r_2)$, and similarly there will be $\phi c_2-\rec{\sqrt{5}}(c_2+r_2)$ values of $b$ satisfying both \eqref{D-Pb} and $b>\rec{\sqrt{5}}(c_2+r_2)$ (this logic would fail if, say, $\rec{\sqrt{5}}(c_2+r_2)>\phi c_2$ because the quantity $\phi c_2-\rec{\sqrt{5}}(c_2+r_2)$ would be negative). In total then the count for this range will be  (again using $\phi+\phi^{-1}=\sqrt{5}$),
	\begin{align}
	\sum_{b=\phi^{-1}r_2}^{\rec{\sqrt{5}}(c_2+r_2)} &\phi b-r_2+\sum_{\rec{\sqrt{5}}(c_2+r_2)}^{\phi c_2} c_2-\phi^{-1}b\label{D-P1}\\ = &\rec{10}\phi (c_2+r_2)^2-\half \phi^{-1}r_2^2-\rec{\sqrt{5}} r_2c_2-\rec{\sqrt{5}} r_2^2+\phi^{-1}r_2^2\nonumber\\ &+\phi c_2^2-\rec{\sqrt{5}} c_2^2-\rec{\sqrt{5}} r_2c_2-\half \phi c_2^2+\rec{10}\phi^{-1}(c_2+r_2)^2+O(\phi^{p/2})\nonumber\\
	= &\rec{2\sqrt{5}} (c_2+r_2)^2-\rec{\sqrt{5}}(c_2^2+2r_2c_2+r_2^2)+\half \phi c_2^2+\half \phi^{-1}r_2^2+O(\phi^{p/2})\nonumber\\ 
	= &-\rec{2\sqrt{5}} (c_2+r_2)^2+\half\phi c_2^2+\half \phi^{-1}r_2^2+O(\phi^{p/2})\label{D-P2}.
	\end{align}
	Plugging in $c_2=c\phi^{p/2-1}$ and $r_2=r\phi^{p/2+1}$ into \eqref{D-P2} gives
	{\small\begin{align}
		h(p/2+1)&=\l( -\rec{2\sqrt{5}}+\half\phi\r) c^2\phi^{p-2}-\rec{\sqrt{5}}rc\phi^p+\l(-\rec{2\sqrt{5}}+\half\phi^{-1}\r) r^2\phi^{p+2}+O(\phi^{p/2})\nonumber\\ 
		&= \l(\rec{2\sqrt{5}} c^2-\rec{\sqrt{5}}rc+\rec{2\sqrt{5}} r^2\r)\phi^p+O(\phi^{p/2})\tr{ for }c\ge r, \label{D-P2B}
	\end{align}}
	where we used $-\rec{\sqrt{5}}\phi^{-2}+\phi^{-1}=-\rec{\sqrt{5}}\phi^2+\phi=\rec{\sqrt{5}}$.
	
	Next consider $t=(p+1)/2$.  In this case the upper bound of \eqref{D-Pb} reduces to $\rec{\sqrt{5}}\phi^{(p+1)/2}$.  As in the previous case, the bound of \eqref{D-Pa} is determined by whether or not we have $b\le \rec{\sqrt{5}}(c\phi^{(p-1)/2}+r\phi^{(p+1)/2})$.  Note that \[c\phi^{(p-1)/2}+r\phi^{(p+1)/2}-(\phi+\phi^{-1})r\phi^{(p-1)/2}=(c-\phi^{-1}r)\phi^{(p-1)/2},\] and this value is non-negative by our bounds on $c$ and $r$.  We conclude that $r\phi^{(p-1)/2}\le \rec{\sqrt{5}}(c\phi^{(p-1)/2}+r\phi^{(p+1)/2})$, and hence there will exist $b$ in the corresponding range.  If  $c> (1-r)\phi$, then  $\rec{\sqrt{5}}(c\phi^{(p-1)/2}+r\phi^{(p+1)/2})>\rec{\sqrt{5}}\phi^{(p+1)/2}$, so in this case the bound of \eqref{D-Pb} always reduces to $\phi b-r\phi^{(p+1)/2}$ and we have a count of 
	\begin{align*}
	\sum_{r\phi^{(p-1)/2}}^{\rec{\sqrt{5}}\phi^{(p+1)/2}}\phi b-r\phi^{(p+1)/2}=\rec{10}\phi^{p+2}-\half r^2\phi^p-\rec{\sqrt{5}}r\phi^{p+1}+r^2\phi^p+O(\phi^{p/2}), 
	\end{align*}
	and hence
	\begin{equation}
		h((p+1)/2)=\l(\half r^2-\f{1}{\sqrt{5}}\phi r+\rec{10}\phi^2\r)\phi^p+O(\phi^{p/2})\tr{ for }c>(1-r)\phi. \label{D-P1B}
	\end{equation}
	If $c\le (1-r)\phi$, then the count becomes almost the same as that of \eqref{D-P1} except $c_2,r_2$ are replaced with $c_1,r_1$ and the upper bound of the last sum becomes $\rec{\sqrt{5}}\phi^{(p+1)/2}$.  Effectively then the count becomes \eqref{D-P2} after replacing $c_2,r_2$ with $c_1,r_1$ and adding
	\begin{align*}
		\sum_{\phi c_1}^{\rec{\sqrt{5}}\phi^{(p+1)/2}}c_1-\phi^{-1}b&=\rs c_1\phi^{(p+1)/2}-\phi c_1^2-\rec{10} \phi^{p}+\half \phi c_1^2+O(\phi^{p/2}),
	\end{align*}
	so in total we have a count of
	\begin{align}
	-\rec{2\sqrt{5}} (c_1+r_1)^2+\rs c_1\phi^{(p+1)/2}+\half \phi^{-1}r_1^2-\rec{10} \phi^{p}+O(\phi^{p/2}).\label{D-P3}
	\end{align}
	
	Plugging in $c_1=c\phi^{(p-1)/2}$ and $r_1=r\phi^{(p+1)/2}$ gives
	{\small \begin{equation}
		h((p+1)/2)=\l(-\rec{2\sqrt{5}}\phi^{-1}c^2+\f{1-r}{\sqrt{5}}c+\rec{2\sqrt{5}}\phi^{-1}r^2-\rec{10}\r)\phi^p\tr{ for }c\le (1-r)\phi, \label{D-P1S}
	\end{equation}}
	where we used $-\rec{2\sqrt{5}}\phi+\half =\rec{2\sqrt{5}}\phi^{-1}$ to get the coefficient for $r^2$.
	
	Finally, consider $t=p/2$.  In this case the upper bound of \eqref{D-Pb} reduces to $\rec{\sqrt{5}}\phi^{p/2}$.  The bound of \eqref{D-Pa} is determined by whether or not we have $b\le \rec{\sqrt{5}}(c+r)\phi^{p/2}$.  As before we find that this quantity is always at least $r\phi^{p/2-1}$.  If $c>1-r$, then this cutoff value is larger than $\rec{\sqrt{5}}\phi^{p/2}$, so \eqref{D-Pa} always reduces to $\phi b-r\phi^{p/2}$ and we get a count of
	\begin{align*}
	\sum_{r\phi^{p/2-1}}^{\rec{\sqrt{5}}\phi^{p/2}} \phi b-r\phi^{p/2}=\rec{10}\phi^{p+1}-\half r^2\phi^{p-1}-\rec{\sqrt{5}}r\phi^p+r^2\phi^{p-1}+O(\phi^{p/2}),
	\end{align*}
	so
	\begin{equation}
		h(p/2)=\l(\half \phi^{-1}r^2-\rec{\sqrt{5}}r+\rec{10}\phi\r)\phi^p+O(\phi^{p/2})\tr{ for }c>1-r. \label{D-P0B}
	\end{equation}
	If instead $c\le 1-r$, then effectively in the same way we derived \eqref{D-P3}, we find our count to be
	\begin{align*}
	-\rec{2\sqrt{5}} (c_0+r_0)^2+\rs c_0\phi^{p/2}+\half \phi^{-1}r_0^2-\rec{10} \phi^{p-1}+O(\phi^{p/2}),
	\end{align*}
	and by plugging in $c_0=c\phi^{p/2}$ and $r_0=r\phi^{p/2}$ we find
	\begin{equation}
		h(p/2)=\l(-\rec{2\sqrt{5}}c^2+\f{1-r}{\sqrt{5}}c+\rec{2\sqrt{5}}\phi^{-2}r^2-\rec{10}\phi^{-1}\r)\phi^p\tr{ for }c\le 1-r, \label{D-P0S}
	\end{equation}
	where we used $-\rec{2\sqrt{5}}+\half\phi^{-1} =\rec{2\sqrt{5}}\phi^{-2}$.
	
	Let $H(n,d)=\rec{n}\sum h(t)$.  As we have argued before, $P(n,d)$ is within $O(\phi^p)$ of $H(n,d)$.  Recall that $h(t)=0$ for $t\ge (p+3)/2$.  First consider $p$ odd, in which case we have by \eqref{D-Pgen}
	\begin{align}
		\sum_{t\le (p-1)/2} h(t)=\l(\half \phi^{-1}r^2-\rec{\sqrt{5}}r+\rec{10}\phi\r)\sum_{t\le(p-1)/2}\phi^{2t}+O(\phi^{p/2})\nonumber\\ 
		=\l(\half \phi^{-1}r^2-\rec{\sqrt{5}}r+\rec{10}\phi\r)\phi^p+O(\phi^{p/2})\label{D-Podd},
	\end{align}
	where we implicitly used $\phi^2-1=\phi$.  If $c>(1-r)\phi$ we add \eqref{D-P1B} (which is effectively just $\phi$ times \eqref{D-Podd}) to \eqref{D-Podd}, use that $1+\phi=\phi^2$, and divide by $\rec{\sqrt{5}}c\phi^p$ to get
	\[
		H(n,d)=\l(\f{\sqrt{5}}{2} \phi r^2-\phi^2 r+\rec{2\sqrt{5}}\phi^3\r)c^{-1}+O(\phi^{-p/2}).
	\]
	If $c\le (1-r)\phi$, then we add \eqref{D-P1S} to \eqref{D-Podd} and get
	\[
		H(n,d)=-\half \phi^{-1}c+(1-r)+\l( r^2-r+\rec{2\sqrt{5}}\phi^{-1}\r)c^{-1}+O(\phi^{-p/2}).
	\]
	
	Now assume $p$ is even, in which case we have
	\begin{align}
	\sum_{t\le (p-2)/2} h(t)=\l(\half \phi^{-1}r^2-\rec{\sqrt{5}}r+\rec{10}\phi\r)\sum_{t\le(p-2)/2}\phi^{2t}+O(\phi^{p/2})\nonumber\\ 
	=\l(\half \phi^{-2}r^2-\rec{\sqrt{5}}\phi^{-1}r+\rec{10}\r)\phi^p+O(\phi^{p/2}).\label{D-Peven}
	\end{align}
	
	If $c\le 1-r\le r$, then $h(p/2+1)=0$ and we only need to add \eqref{D-P0S} to find
	\[
		H(n,d)= -\half c+(1-r)+\l( \phi^{-1}r^2-\phi^{-1}r+\rec{2\sqrt{5}}\phi^{-2}\r)c^{-1}+O(\phi^{-p/2}).
	\]
	
	If $1-r\le c\le r$, then we add \eqref{D-P0B} to \eqref{D-Peven} to find
	\begin{equation}
		H(n,d)=\l(\f{\sqrt{5}}{2} r^2-\phi r+\rec{2\sqrt{5}}\phi^2\r)c^{-1}+O(\phi^{-p/2}). \label{D-PTemp}
	\end{equation}
	
	Finally, if $c\ge r\ge 1-r$ we add \eqref{D-P2B} divided by $\rec{\sqrt{5}}c\phi^p$ to \eqref{D-PTemp} to find
	\[
		H(n,d)=\f{1}{2}c-r+\l(\phi r^2-\phi r+\rec{2\sqrt{5}}\phi^2\r)c^{-1}+O(\phi^{-p/2}).
	\]
	
	Plugging in $r=1-d$ gives us the desired result by observing that, for example, $r^2-r=d^2-d$ and \[\f{\sqrt{5}}{2}r^2-\phi r+\f{2}{\sqrt{5}}\phi^2=\f{\sqrt{5}}{2}\l(r-\f{1}{\sqrt{5}}\phi\r)^2=\f{\sqrt{5}}{2}\l(d-\f{1}{\sqrt{5}}\phi^{-1}\r)^2.\]
\end{proof}


We conclude this section by showing that $n$ with two $n$-good pairs always behave  ``as expected.''

\begin{prop}\label{P-Double}
	Let $n\ge 2$ be such that there exists two $n$-good pairs.  Then $n$ is not $\half$-paradoxical. That is, $n\in D$ if and only if $N(\phi n)=\floor{\phi n}$.
\end{prop}
\begin{proof}
	Let $a,b,t$ be as in Theorem~\ref{T-Char}, noting that our assumption on $n$ implies that $a>f_{t-1}$.  Thus if $t\ge 2$ is even, we have by Lemma~\ref{L-abDel}
	\[
	\del_n\le \phi^{-t}(\phi f_t-f_{t-1}-1)\le \rec{\sqrt{5}}(\phi-\phi^{-1}+2\phi^{1-2t}-\phi^{-t})\le \rec{\sqrt{5}}(\phi-\phi^{-1}+.1)<.5.
	\]
	Because $\del_n<.5$, we conclude that $N(\phi n)=\floor{\phi n}$.  Similarly we find that $t$ odd implies $N(\phi n)=\ceil{\phi n}$.  As $n\in D$ if and only if $t$ is even, we conclude the result.
\end{proof}
\section{Gap Sizes}\label{S-Gaps}
In this section we prove Theorems~\ref{T-Gap} and \ref{T-Gap2}, and to this end we will need some additional results concerning $\del_n$ and $\Del_n$.

\begin{lem}\label{L-Add}
	For any $n,\ell\ge 0$, we have
	\begin{align*}
	\del_{n+\ell}&=\del_n+\del_\ell-\floor{\del_n+\del_\ell},\\ 
	\Del_{n+\ell}&=\Del_n+\Del_{\ell}+\ceil{-\Del_n-\Del_{\ell}}.
	\end{align*}
\end{lem}
\begin{proof}
	By definition, \begin{align*}\del_{n+\ell}&=\phi (n+\ell)-\floor{\phi n+\phi\ell}\\ &=\phi(n+\ell)-\floor{\floor{n}+\del_n+\floor{\ell}+\del_\ell}\nonumber\\ &=\phi n+\phi \ell-\floor{n}-\floor{\ell}-\floor{\del_n+\del_\ell}\nonumber \\ &=\del_n+\del_\ell-\floor{\del_n+\del_\ell}.\label{E-Del}\end{align*}
	
	The proof for $\Del_{n+\ell}$ is essentially the same.
\end{proof}

\begin{lem}\label{L-ConFloor}
	For every $n$ there exists $1\le \ell\le 5$ such that $\del_{n+\ell}<\rec{\sqrt{5}}\phi^{-1}$, and one can choose $\ell\le 3$ unless $\del_n\in [0,.15)\cup (.65,.77)$. Moreover, there exists $1\le \ell<\ell'\le 6$ such that $\del_{n+\ell},\del_{n+\ell'}<\rec{\sqrt{5}}\phi^{-1}$ unless $\del_n\in[0,.3)\cup(.65,1]$.
	
	For every $n$ there exists $1\le \ell\le 5$ such that $\Del_{n+\ell}<\rec{\sqrt{5}}\phi^{-1}$, and one can choose $\ell\le 3$ unless $\Del_n\in (.13,.24)\cup (.51,.62)$. Moreover, there exists $1\le \ell<\ell'\le 6$ such that $\Del_{n+\ell},\Del_{n+\ell'}<\rec{\sqrt{5}}\phi^{-1}$ unless $\Del_n\in [0,.24)\cup (.36,.62)\cup (.97,1]$.
\end{lem}

\begin{proof}
	Define $I_\ell:=\{x\in[0,1]:x+\del_\ell-\floor{x+\del_\ell}<\rec{\sqrt{5}}\phi^{-1}\}$.  Equivalently, one can view this as the set of $x$ such that there exists some $k\in \Z$ such that $k-\del_\ell\le x<k+\rec{\sqrt{5}}\phi^{-1}-\del_\ell$.  Computing these values we find
	\begin{align*}
	I_1&=[.38\ldots,.65\ldots),\ I_2=[0,.04\ldots)\cup [.76\ldots,1],\ I_3=[.14\ldots,.42\ldots),\\  I_4&=[.52\ldots,.80\ldots),\ I_5=[0,.18\ldots)\cup (.90\ldots,1],\ I_6=(.29\ldots,.56\ldots).
	\end{align*}
	One can verify that for every $x\in[0,1]$ there exists some $\ell\le 5$ such that $x\in I_\ell$, and moreover that one can take $\ell\le 3$ provided $\del_n\notin [0,.15)\cup (.65,.77)$.  Note that we can make these intervals slightly smaller, but we have no need to do so.  If $x=\del_n$, then  Lemma~\ref{L-Add} implies that $\del_{n+\ell}<\rec{\sqrt{5}}\phi^{-1}$ for this choice of $\ell$.  Further, given any $x\notin [0,.29)\cup (.65,.91)$, one can find  $1\le \ell<\ell'\le 6$ such that $x\in I_\ell,I_{\ell'}$, and by considering $x=\del_n$ we conclude the first part of the lemma.
	
	For the second part, define $J_\ell=\{x\in[0,1]:x+\Del_\ell+\ceil{-x-\Del_\ell}\}$.  We find
	\begin{align*}
	J_1&=(.61,\ldots,.89\ldots),\ J_2=(.23\ldots,.51\ldots),\ J_3=[0,.13\ldots)\cup (.85\ldots,1],\\ 
	J_4&=(.47\ldots,.74\ldots),\ J_5=(.09\ldots,.36\ldots),\ J_6=(.70\ldots,.98\ldots).
	\end{align*}
	With this one can argue essentially as before to conclude the results concerning $\Del_n$.
\end{proof}

We are now just about ready to prove Theorem~\ref{T-Gap}. Our proof will utilize  d'Ocagne's identity \cite{V,W},
\[
f_{m+1}f_{n}-f_{m}f_{n+1}=(-1)^{m}f_{n-m},
\]
which is a generalization of Cassini's identity.
We will also make use of the inequality
\begin{equation} f_t\le (1+\phi^{-8})\rec{\sqrt{5}}\phi^t\le .46\phi^t,\label{E-In}\end{equation}
which is valid for $t\ge 4$.
\begin{proof}[Proof of Theorem~\ref{T-Gap}]
	We demonstrated before the statement of the theorem that each of these gap sizes can occur, so it remains to show that these are the only values that occur.  We first consider $d_{k+1}-d_k$ and let $n=d_k$.  One can verify that the statement is true for $n\le f_6f_5=40$, so we can assume $n>40$, and hence $t(m)\ge 6$ for all $m\ge n$ by Lemma~\ref{L-Chicken}.
	
	By Lemma~\ref{L-ConFloor}, there exists some $1\le \ell\le 5$ such that $\del_{n+\ell}<\rec{\sqrt{5}}\phi^{-1}$. This implies that $n+\ell\in D$ by Proposition~\ref{P-Thresh}, and hence \[d_{k+1}-d_k\le n+\ell-n=\ell\le 5.\]  It remains to show that this difference is not 4. Assume for the sake of contradiction that $d_{k+1}=n+4$, which implies that $t(n+4)$ is even and that $t(n+\ell)$ is odd for all $1\le \ell\le 3$.  By Lemma~\ref{L-ConFloor} we must have that $\del_{n}\in [0,.15)\cup (.65,.77)$.  
	
	First consider $\del_n\in[0,.15)$, which by Lemma~\ref{L-Add} implies $\del_{n+4}\in(.47,.63)$.  Let $a=a(n+4),b=b(n+4),t=t(n+4)$ be as in Theorem~\ref{T-Char}, and recall that we are assuming $t$ to be even and at least 6.  We claim that $b-a> f_{t-2}$.  Indeed, if instead $a\ge b-f_{t-2}$, then by our bounds on $\del_{n+4}$, Lemma~\ref{L-abDel}, and \eqref{E-In}, we would have
	\begin{align*}
	.47\phi^t&\le \del_n \phi^t=\phi b-a\le \phi b-b+ f_{t-2}=\phi^{-1}b+f_{t-2}\\ &\le \phi^{-1}f_t+ f_{t-2}\le .46(\phi^{t-1}+\phi^{t-2})=.46\phi^t,
	\end{align*}
	a contradiction.

	By d'Ocagne's identity with $n=t-1,m=t-4$ and the assumption that $t$ is even, we find
	\begin{align*}
	n+2=(n+4)-2&=af_t+bf_{t-1}+f_{t-4}f_t-f_{t-3}f_{t-1}\\ &=(a+f_{t-4})f_t+(b-f_{t-3})f_{t-1}.
	\end{align*}
	With this in mind, define $a':=a+f_{t-4}$ and $b':=b-f_{t-3}$.  Note that $a'\ge a\ge 1$ and $b'\le b\le f_t$.  Moreover, \[b'-a'=b-f_{t-3}-a-f_{t-4}=b-a-f_{t-2}> 0.\] We conclude by the uniqueness of the integers of Theorem~\ref{T-Char} that $a(n+2)=a',$ $b(n+2)=b'$, and $t(n+2)=t$.  In particular, $t(n+2)$ is even, a contradiction to our assumption that $n+2\notin D$.
	
	Now assume $\del_n\in (.65,.77)$.  Let $a=a(n),b=b(n),t=t(n)$, and recall that $n=d_k$ implies that $t$ is even.  We claim that $b-a> f_{t-1}$.  If this were not the case, then as before we find
	\[
	.65\phi^t\le \phi b-a\le \phi^{-1}f_t+f_{t-1}\le .46\cdot 2 \phi^{t-1}\le .57\phi^t,
	\]
	a contradiction. By Cassini's identity and the assumption that $t$ is even, we find
	\[
	n+1=af_t+bf_{t-1}+f_{t-3}f_t-f_{t-2}f_{t}=(a+f_{t-3})f_t+(b-f_{t-2}):=a'f_t+b'f_{t-1}.
	\]
	Again $a'\ge a\ge 1,b'\le b\le f_t$, and $b'-a'=b-a-f_{t-1}\ge 0$, so we conclude that $t(n+1)=t$, a contradiction to our assumption that $n+1\notin D$.
	
	The proof for the case $u_{k+1}-u_k$ is essentially the same, so we only sketch the details.  One can show that $u_{k+1}-u_k\le 5$ exactly as we did for $d_{k+1}-d_k$. Let $n=u_k$, and again we can verify the statement for $n\le 40$.  Assume $n+4=u_{k+1}$.  Lemma~\ref{L-ConFloor} shows that this implies $\Del_n\in (.13,.24)\cup (.51,.62)$.
	
	If $\Del_n\in (.13,.24)$, then Lemma~\ref{L-Add} implies that $\Del_{n+4}\in(.65,.77)$.  Let $a=a(n+4),b=b(n+4),t=t(n+4)$.  By noting that $t$ is odd and using essentially the same argument as in the case $\del_n\in(.65,.77)$ (using the exact same numerical bounds), we find that $b-a\ge f_{t-1}$.  Because $t$ is odd we have by Cassini's identity
	\[
		n+3=af_t+bf_{t-1}+f_{t-3}f_t-f_{t-2}f_{t-1}=(a+f_{t-3})f_t+(b-f_{t-2})f_t.
	\]
	As before we find $t(n+3)=t$, a contradiction.
	
	If $\Del_n\in (.51,.62)$, let $a=a(n),b=b(n),t=t(n)$.  Doing the exact same computations as in the case $\del_{n}\in[0,.15)$, we find $b-a\ge f_{t-2}$, and hence by d'Ocagne's identity and the assumption that $t$ is odd we have
	\[
		n+2=(a+f_{t-3})f_t+(b-f_{t-4})f_{t-1},
	\]
	and from this we conclude $t(n+2)=t$, a contradiction.
\end{proof}

We prove Theorem~\ref{T-Gap2} using similar ideas.
\begin{proof}[Proof of Theorem~\ref{T-Gap2}]
	Let $S=\{2,3,4,5,6,8,10\}$, $D_2=\{d_{k+2}-d_k\}$, and $U_2=\{u_{k+2}-u_k\}$.  We first show that $S\sub D_2,U_2$.  One can see from \eqref{D1} that $2,3,4,5,8\in D_2$ (using $48-46,\ 10-7,\ 9-5,\ 7-2,\ 23-15$), and one can verify that $d_{53}-d_{51}=102-96=6$ and $d_{374}-d_{372}=756-746=10$.  Similarly from \eqref{U1} one sees that $2,3,4,5,6,8\in U_2$ (using $21-19,\ 6-3,\ 8-4,\ 11-6,\ 14-8,\ 50-42$), and one can verify that $u_{961}-u_{959}=1927-1917=10$.  It remains to show $D_2,U_2\sub S$.  
	
	Note that we can not have, for example, \[9=d_{k+2}-d_k=(d_{k+2}-d_{k+1})+(d_{k+1}-d_k)\] without one of these differences being 4 or larger than 5, so by Theorem~\ref{T-Gap} we conclude that $9\notin D_2,U_2$.  Similarly $m\notin D_2,U_2$ for any $m>10$. Trivially $1\notin D_2,U_2$, so in order to show $D_2,U_2\sub S$ it remains to prove $7\notin D_2,U_2$.
	
	Let $n=d_k$.  One can verify that the statement is true up to $n\le f_8f_7=273$, so we can assume $n>273$, and hence $t(m)\ge 8$ for all $m\ge n$ by Lemma~\ref{L-Chicken}.  Note that we have $d_{k+2}-d_k\le 6$ whenever $\del_n\notin [0,.29)\cup (.65,1]$ by Lemma~\ref{L-ConFloor}, so it remains to deal with these cases.  As we will see, most of these cases can be solved automatically by our previous propositions and theorems.  
	
	If $\del_n\in(.73,1]\sub (1-\rec{\sqrt{5}}\phi^{-1},1]$, then $n\in U$ by Proposition~\ref{P-Thresh}, which contradicts the assumption $n=d_k\in D$.  If $\del_n\in (.65,.74)$, then one can verify that $\del_{n+4}<\rec{\sqrt{5}}\phi^{-1}$ (in essentially the same way by which we constructed $I_4$ in Lemma~\ref{L-ConFloor}), so $n+4\in D$.  But $n+4\ne d_{k+1}$ by Theorem~\ref{T-Gap}, so $d_{k+2}\le n+4$ and we conclude the result in this case.  Similarly, if $\del_n\in (.14,.29)$, then $\del_{n+3}<\rec{\sqrt{5}}\phi^{-1}$ and we have $n+3\in D$.  If $d_{k+2}\le n+3$ then we are done, and otherwise $d_{k+1}=n+3$ and we can not have $d_{k+2}=d_k+7=d_{k+1}+4$ by Theorem~\ref{T-Gap}.  If $\del_n\in[0,.04)$, then $\del_{n+2},\del_{n+5}<\rec{\sqrt{5}}\phi^{-1}$, so $d_{k+2}\le n+5$ in this case.
	
	The final case to consider is $\del_n\in (.039,.15)$.  Assume that in this range we have $d_{k+2}=n+7$, and let $a=a(n+7),b=b(n+7),t=(n+7)$.  Note that in this range we have $\del_{n+7}\in(.36,.48)$.  We claim that $b-a>f_{t-4}$.  Indeed if this were not the case, then by \eqref{E-In} (and the assumption $t-4\ge 4$), we would have
	\[
	.36\phi^t\le \phi^{-1}b+f_{t-4}\le .46(\phi^{t-1}+\phi^{t-4})\le .35\phi^t,
	\]
	a contradiction.  By d'Ocagne's identity with $m=t-6,n=t-1$ and $t$ even, we find
	\[
	n+2=af_t+bf_{t-1}+f_{t-6}f_t-f_{t-5}f_{t-1}=(a+f_{t-6})+(b-f_{t-5}).
	\]
	As in the proof of Theorem~\ref{T-Gap} we conclude $t(n+2)=t$ is even, so $d_{k+1}\le n+2$.  In this range we also have $\del_{n+5}<\rec{\sqrt{5}}\phi^{-1}$, so $d_{k+2}\le n+5$, a contradiction.  We conclude the result for $D_2$.
	
	We now turn to $u_{k+2}-u_k$.  Let $n=u_k$, and again we can verify this theorem for $n\le 273$ and hence can assume that $t(m)\ge 8$ for all $m\ge n$.  We will be done if $\del_n\notin [0,.24)\cup (.36,.62)\cup (.97,1]$ by Lemma~\ref{L-ConFloor}, so assume this is not the case.
	
	Since $(.97,1]\sub (1-\rec{\sqrt{5}}\phi^{-1},1]$, this is dealt with by Proposition~\ref{P-Thresh} as in the down-integer case.  Similarly if $\Del_n\in (.12,.24)$, then $\Del_{n+7}<\rec{\sqrt{5}}\phi^{-1}$, and hence $n+7\notin D$.  If $\Del_n\in [0,.13)$, then $\Del_{n+3}<\rec{\sqrt{5}}\phi^{-1}$ and hence either $u_{k+2}\le n+3$ or $u_{k+1}=n+3$, in which case we can not have $u_{k+2}=(n+3)+4$.  If $\Del_n\in(.47,.62)$, then $\Del_{n+4}<\rec{\sqrt{5}}\phi^{-1}$ and $n+4\in D$, but $u_{k+1}\ne n+4$ so we have $u_{k+2}\le n+4$.
	
	The last case to consider is $\Del_n\in (.36,.48)$.  Let $a=a(n),b=b(n),t=t(n)$.  By the exact same computations we did for the case $\del_n\in (.039,.15)$, we find $b-a>f_{t-4}$.  By d'Ocagne's identity and the assumption that $t$ is even we find
	\[
	n+5=af_t+bf_{t-1}+f_{t-6}f_t-f_{t-5}f_{t-1}=(a+f_{t-6})+(b-f_{t-5}).
	\]
	and conclude that $t(n+5)=t$ is even.  In this range we have $\Del_{n+2}<\rec{\sqrt{5}}\phi^{-1}$, so $u_{k+2}\le n+5$, finishing the proof.
\end{proof}

\section{Another View: Starting from the End}\label{S-End} Here we describe another approach for determining $n$-good pairs. This is by means of what we call
a {\it reverse} Fibonacci walk $R = (r_1, r_2,  r_3, r_4, \ldots )$. To perform this walk, we start with values $r_1 = n$ and  $r_2 = b $, where $1 \leq b < n$,
and define $r_{k+2} = r_k - r_{k+1}$. The walk continues as long as $r_k > 0$. Let $t = t(n,b)$ be the largest index for which $r_t > 0$, and we note that in general this value $t$ does not equal $t(n)$ as defined in Theorem~\ref{T-Char}. Thus,
$R = R(n,b) = (n, b, n-b, 2b-n, 2n - 3b, 5b - 3n, 5n - 8b, \ldots)$. In general,
\begin{align} \label{r_k}
r_k = \begin{cases}
f_{k-2}n - f_{k-1}b~~~ \text{if $k$ is odd},\\
f_{k-1}b - f_{k-2}n ~~~\text{if $k$ is even}.
\end{cases}
\end{align}
For example,
\begin{align*}
R(100, 61) &= (100, 61, 39, 22, 17, 5, 12),\\
R(100, 62) &= (100, 62, 38, 24, 14, 10, 4, 6),\\
R(100, 63) &= (100, 63, 37, 26, 11, 15).
\end{align*}
It follows from the definition of $R$ that $w_k(5,12),\ w_k(4,6)$ and $w_k(11,15)$ are all Fibonacci walks that hit $n=100$.
Of course, any pair $w_k(r_{k-1}, r_k)$ in $R(n,b)$ has this property, but only the terminal pair $(r_{t-1}, r_t)$ has a chance of having length $s(n)$, that is, of generating an $n$-slow Fibonacci walk. Thus,
to determine $s(n)$, we only have to try a linear number of $b$'s as opposed to the quadratic number of candidates needed for the ordinary Fibonacci walk for $n$.  In fact, we can do much better than this.

\begin{prop}
	There exists an $O(\log n)$ algorithm that takes $n$ as input and returns all $n$-good pairs.
\end{prop}
\begin{proof}
	The algorithm proceeds as follows.  One first generates the reverse Fibonacci walks $R(\floor{\phi n},n)$ and $R(\ceil{\phi n},n)$, and it is not too difficult to see that we need to generate $O(\log n)$ terms for each of these sequences.  By Theorem~\ref{T-Char}, the pair $(r_{t-1},r_t)$ of the longer sequence will be the pair $(a(n),b(n))$, and from this information one can deduce all $n$-good pairs by Theorem~\ref{T-Char}.
\end{proof}

Let us rescale the reverse Fibonacci walk as 
\begin{align*}
R(1,\tfrac{b}{n})& = (1,\tfrac{b}{n}, 1 - \tfrac{b}{n}, \tfrac{2b}{n} - 1, 2 - \tfrac{3b}{n}, \tfrac{5b}{n} - 3, 5 - \tfrac{8b}{n}, \ldots)\\
&= (\rho_1, \rho_2, \rho_3, \rho_4, \ldots, \rho_t)
\end{align*}
where $\rho_k = \tfrac{r_k}{n}$ and $t = t(n,b)$ is the number of terms in the walk. The following result is well-known (e.g., see \cite{W}):\\
{\bf Fact}.
{\small \begin{align*}
0 = \tfrac{f_0}{f_1} < \tfrac{f_2}{f_3}< \cdots <  \tfrac{f_{2k-2}}{f_{2k-1}} <\tfrac{f_{2k}}{f_{2k+1}}<  \cdots < \tfrac{1}{\phi} < \cdots 
<\tfrac{f_{2k+1}}{f_{2k+2}}< \tfrac{f_{2k-1}}{f_{2k}} < \cdots < \tfrac{f_3}{f_4}<\tfrac{f_1}{f_2}=1.
\end{align*}}
Define
\begin{align} \label{K-def}
K_{2u+1} = {\Big(} \tfrac{f_{2u-2}}{f_{2u-1}}, \tfrac{f_{2u}}{f_{2u+1}}{\Big]},~~~~K_{2u+2} = {\text{$\Big[$}}\tfrac{f_{2u+1}}{f_{2u+2}}, \tfrac{f_{2u-1}}{f_{2u}}\Big),
\end{align}
for $u \geq 1$. Thus, $(0,1) = \cup_{u \geq 1}K_u $ is a decomposition of $(0,1)$ into disjoint half-open intervals. We can picture this as
\begin{align*}
(0,1) = K_3 K_5 K_7 \ldots K_{2k-1} K_{2k+1} \ldots \tfrac{1}{\phi} \ldots K_{2k+2} K_{2k} \ldots K_6 K_4.
\end{align*}
with the point  $\tfrac{1}{\phi}= 0.618 \ldots$ separating the odd $K$'s from the even $K$'s. Note that the lengths of the
$K$'s decrease exponentially rapidly since $K_m$ has length $\tfrac{1}{f_{m-2}f_m} \sim \tfrac{5}{\phi^{2m-2}}$.
\begin{prop} \label{K-proposition} $ t(n,b) = m \iff \tfrac{b}{n} \in K_m$.
\end{prop}
\proof There are two cases depending on the parity of $m$. First, suppose $m = 2u+2$ and that $\f{b}{n}\in K_m$. Thus,
\begin{align*}
\frac{f_{2u+1}}{f_{2u+2}} \leq \frac{b}{n} &<\frac{f_{2u-1}}{f_{2u}}.
\end{align*}
Hence,
\begin{align*}
nf_{2u+1} - bf_{2u+2} &\leq 0 ~~~~\text{and}~~~~nf_{2u-1}-bf_{2u} > 0.
\end{align*}
This implies that $r_{2u+3} \leq 0$ and $r_{k} > 0$ for $k \leq 2u+1$. However, we claim that $r_{2u+2}=bf_{2u+1} - nf_{2u} > 0$ as well since
\begin{align*}
\frac{b}{n} \geq \frac{f_{2u+1}}{f_{2u+2}} > \frac{1}{\phi} > \frac{f_{2u}}{f_{2u+1}}.
\end{align*}
Consequently, $t(n,b) = 2u+2$, as claimed. The argument for $m$ odd is similar and is omitted. \qed

With this we can imagine a point starting at 0 hopping along with steps of size $\tfrac{1}{n}$ until it gets to the point 1.  By Proposition~\ref{K-proposition}, finding the value $b$ with $t(n,b)$ as large as possible is thus equivalent to knowing which step $b$ in our 0 to 1 walk will land in the interval $K_m$ with the largest value of $m$. Sometimes there may be two
consecutive values of $b$ which do this (but never three!). Also, it isn't necessarily the step which is closest to $\tfrac{1}{\phi}$
because of the non-symmetry of the $K$'s.

(The argument that there can't be three consecutive values of $b$ that land in the same $K_m$ with the largest value of $m$ goes as follows.
Suppose that our hopping point lands three times in such a $K_{2u+1}$. Since the length of $K_{2u+1}$ is $\tfrac{1}{f_{2u-1} f_{2u+1}}$,
the step size must be less than $\tfrac{1}{2}$ of this. In order for this $K$ to be optimal, the {\it next} step must go beyond $\tfrac{1}{\phi}$ and 
in fact beyond $K_{2u+2}$. Since the right-hand boundary point of $K_{2u+2}$ is $\tfrac{f_{2u-1}}{f_{2u}}$ this implies that
\begin{align*}
\frac{1}{2 f_{2u-1} f_{2u+1}} > \frac{f_{2u-1}}{f_{2u}}-\frac{f_{2u}}{f_{2u+1}} = \frac{1}{f_{2u} f_{2u+1}} ,
\end{align*}
in other words
\begin{align*}
f_{2u} > 2 f_{2u-1}
\end{align*}
which is impossible. The argument for the other parity is similar.)

\section{Magic tricks.}\label{S-Mag} Returning to some of the original motivation for this study, we describe two similar tricks based on Fibonacci walks. For
the first (somewhat wimpy) trick,
a spectator is asked to choose two arbitrary numbers $a$ and $b$ between 1 and 8. They can be equal if desired. Then the spectator is
asked to form a Fibonacci walk starting with $a$ and $b$, and continuing for 5 steps. After a bit of mental calculation, the performer announces
what the number could have been if the spectator had taken a sixth step (and the audience goes wild, or perhaps goes home!).

How does it work? Suppose the fifth step ends with the number $N$. We know that $N = 5a - 8b$. Reducing this modulo 8, we have $5a \equiv N \pmod 8$,
or $a \equiv 5N \pmod 8$. Since $a \leq 8$, then we know $a = 5N \pmod 8$, or 8 if this is 0. $b$ is now $\tfrac{5a - N}{8}$ and the sixth number in the walk
would be $13b - 8a = 5b -3a -N$. (We told you it was wimpy!) The same technique works if the spectator can choose two values $a$ and $b$ between 1 and 13
and is asked to take a Fibonacci walk of six steps. The performer can then predict what the seventh step would have been.  
One can also show the correctness of this trick by using Theorem~\ref{T-Char}.

We mention a (somewhat better) magic trick which was originally given by Richard Stanley and which was generalized by Marc van Leeuwen \cite{S2}.  A spectator chooses two numbers $1\le a_1\le a_2\le N$ and then constructs the sequence defined by $a_k=a_{k-1}+a_{k-2}$ up to some value $k$ with $N<\phi^{k-1}/2+\phi$.  
For example, for $N=25$, the number of additions should be at least 9. The spectator tells the magician $a_k$, and then the magician says that the next number in the sequence would be $N(\phi a_k)$.  The correctness of this trick can be proven using similar ideas as that of Lemma~\ref{L-Form}. The advantage of the first trick is that the calculations involved can be easily performed mentally. This is not true for the second trick (at least for us!).

\section{Concluding Remarks}\label{S-Con}
There are a number of open problems left to consider, many of which concern the gap sizes of $D$ and $U$.  To this end, define \begin{align*}&D_\ell=\{d_{k+\ell}-d_k:k\ge 1\},\\  &D_\ell(m)=\{n:\exists k\ge 1\tr{ such that }n=d_k,\ d_{k+\ell}-n=m\}.\end{align*}  We similarly define $U_\ell$ and $U_\ell(m)$.

We suspect that one can determine $D_\ell$ and $U_\ell$ for any fixed $\ell$ by using the techniques used to prove Theorems~\ref{T-Gap} and \ref{T-Gap2}.
For example, computations suggest that $D_3 = U_3 = \{3, 4, 5, 6, 7, 8, 10, 11, 13\}$. Is this actually the case?
It seems to be more challenging to determine these sets for all $\ell$.

\begin{prob}
	Determine $D_\ell,U_\ell$ for all $\ell$.
\end{prob}

We have seen that $D_1=U_1$ and $D_2=U_2$.  We suspect that this continues to hold.
\begin{con}
	$D_\ell=U_\ell$ for all $\ell$.
\end{con}

Some elements of $D_\ell$ and $U_\ell$ seem to be ``sparser'' than others.  For example, the smallest element of $U_2(10)$ is $1917$, so it seems like having $u_{k+2}-u_k=10$ is fairly rare.  It would be of interest to make this observation more precise.
\begin{prob}
	Determine the densities of the sets $D_\ell(m)$ and $U_\ell(m)$ for various $\ell$ and $m$, and in particular for $\ell=1$.
\end{prob}

We suspect that $D_\ell(m)$ has positive density for all $m\in D_\ell$.  In particular, we suspect that the following is true.
\begin{conj}
	For all $m\in D_\ell$, we have $|D_\ell(m)|=\infty$.
\end{conj}

In this paper we defined $w_k$ to follow a Fibonacci-like recurrence.  More generally, for any sequence $t_k$ satisfying $t_{k+r}=\sum_{i=0}^{r-1} \alpha_it_{k+i}$ for $k\ge 1$ and $\alpha_i\in \N$, one could define an $n$-slow $t_k$-walk to be any sequence $w_k$ of positive integers satisfying the recurrence of $t_k$ and which generates $n$ as slowly as possible.

\begin{quest}
	What can be said about $n$-slow $t_k$-walks for other sequences $t_k$?  In particular, what can be said about the tribonacci sequence which has the recurrence $t_{k+3}=t_{k+2}+t_{k+1}+t_k$?
\end{quest}

Perhaps the most important open problem is the following.
\begin{prob}
	Figure out more magic tricks using Fibonacci walks!
\end{prob}

\newpage

\end{document}